\documentclass[11pt, a4paper]{amsart}
\usepackage[dvips]{graphicx}
\usepackage{amssymb, amstext, amscd, amsmath, amsthm, amsfonts, color}
\usepackage{epstopdf}
\usepackage{epsfig}
\usepackage{geometry} 
\usepackage{tikz,mathdots,enumerate,tabularx}
\usetikzlibrary{ decorations.markings}
\usetikzlibrary{arrows}

\usepackage{pgfplots}% This uses tikz
\usepackage{url}
\usepackage{caption}

\usepackage{kbordermatrix} % include package @ document preamble
\renewcommand{\kbldelim}{[} % change default array delimiters to parentheses
\renewcommand{\kbrdelim}{]}

\usepackage{blindtext}
\usepackage{url}
\usepackage{tikz}
\usetikzlibrary{shapes.geometric, calc}
\usepackage{array}
\usepackage{epsfig}
\usepackage{eucal}
\usepackage{latexsym}
\usepackage{mathrsfs}
\usepackage{textcomp}
\usepackage{verbatim}

\usepackage{hhline}

\usepackage{hyperref}

\begin{document}

%=================================================================
%% Please use the following mathematics environments:

\newtheorem{Theorem}{Theorem}[section]
\newtheorem{Lemma}[Theorem]{Lemma}
\newtheorem{Characterization}[Theorem]{Characterization}
\newtheorem{Proposition}[Theorem]{Proposition}
\newtheorem{Property}[Theorem]{Property}
\newtheorem{Problem}[Theorem]{Problem}
\newtheorem{Example}[Theorem]{Example}
\newtheorem{Remark}[Theorem]{Remark}
\newtheorem{Corollary}[Theorem]{Corollary}
\newtheorem{Definition}[Theorem]{Definition}
%% For proofs, please use the proof environment (the amsthm package is loaded by the MDPI class).

%
%      Blackboard bold letters
\newcommand{\bC}{{\mathbb{C}}}
\newcommand{\bD}{{\mathbb{D}}}
\newcommand{\bN}{{\mathbb{N}}}
\newcommand{\bQ}{{\mathbb{Q}}}
\newcommand{\bR}{{\mathbb{R}}}
\newcommand{\bT}{{\mathbb{T}}}
\newcommand{\bX}{{\mathbb{X}}}
\newcommand{\bZ}{{\mathbb{Z}}}
\newcommand{\bH}{{\mathbb{H}}}
%
%      Capital script letters
  \newcommand{\A}{{\mathcal{A}}}
  \newcommand{\B}{{\mathcal{B}}}
  \newcommand{\C}{{\mathcal{C}}}
  \newcommand{\D}{{\mathcal{D}}}
  \newcommand{\E}{{\mathcal{E}}}
  \newcommand{\F}{{\mathcal{F}}}
  \newcommand{\G}{{\mathcal{G}}}
\renewcommand{\H}{{\mathcal{H}}}
  \newcommand{\I}{{\mathcal{I}}}
  \newcommand{\J}{{\mathcal{J}}}
  \newcommand{\K}{{\mathcal{K}}}
\renewcommand{\L}{{\mathcal{L}}}
  \newcommand{\M}{{\mathcal{M}}}
  \newcommand{\N}{{\mathcal{N}}}
\renewcommand{\O}{{\mathcal{O}}}
\renewcommand{\P}{{\mathcal{P}}}
  \newcommand{\Q}{{\mathcal{Q}}}
  \newcommand{\R}{{\mathcal{R}}}
  \newcommand{\T}{{\mathcal{T}}}
  \newcommand{\U}{{\mathcal{U}}}
  \newcommand{\V}{{\mathcal{V}}}
  \newcommand{\W}{{\mathcal{W}}}
  \newcommand{\X}{{\mathcal{X}}}
  \newcommand{\Y}{{\mathcal{Y}}}
  \newcommand{\Z}{{\mathcal{Z}}}
%
% Fraktur letters
\newcommand{\fA}{{\mathfrak{A}}}
\newcommand{\fB}{{\mathfrak{B}}}
\newcommand{\fC}{{\mathfrak{C}}}
\newcommand{\fD}{{\mathfrak{D}}}
\newcommand{\fE}{{\mathfrak{E}}}
\newcommand{\fF}{{\mathfrak{F}}}
\newcommand{\fG}{{\mathfrak{G}}}
\newcommand{\fH}{{\mathfrak{H}}}
\newcommand{\fI}{{\mathfrak{I}}}
\newcommand{\fJ}{{\mathfrak{J}}}
\newcommand{\fK}{{\mathfrak{K}}}
\newcommand{\fL}{{\mathfrak{L}}}
\newcommand{\fM}{{\mathfrak{M}}}
\newcommand{\fN}{{\mathfrak{N}}}
\newcommand{\fO}{{\mathfrak{O}}}
\newcommand{\fP}{{\mathfrak{P}}}
\newcommand{\fQ}{{\mathfrak{Q}}}
\newcommand{\fR}{{\mathfrak{R}}}
\newcommand{\fS}{{\mathfrak{S}}}
\newcommand{\fT}{{\mathfrak{T}}}
\newcommand{\fU}{{\mathfrak{U}}}
\newcommand{\fV}{{\mathfrak{V}}}
\newcommand{\fW}{{\mathfrak{W}}}
\newcommand{\fX}{{\mathfrak{X}}}
\newcommand{\fY}{{\mathfrak{Y}}}
\newcommand{\fZ}{{\mathfrak{Z}}}
\newcommand{\ul}{\underline  }
% Misc notation
\newcommand{\Aut}{\operatorname{Aut}}
\newcommand{\sgn}{\operatorname{sgn}}
\newcommand{\rank}{\operatorname{rank}}
\newcommand{\adj}{\operatorname{adj}}
\newcommand{\ran}{\operatorname{ran}}
\newcommand{\supp}{\operatorname{supp}}
\newcommand{\conv}{\operatorname{conv}}
\newcommand{\cone}{\operatorname{cone}}
\newcommand{\vspan}{\operatorname{span}}
\newcommand{\proj}{\operatorname{proj}}
\newcommand{\Isom}{\operatorname{Isom}}
\newcommand{\qIsom}{\operatorname{q-Isom}}
\newcommand{\Cknet}{{\mathcal{C}_{\text{knet}}}}
\newcommand{\Ckag}{{\mathcal{C}_{\text{kag}}}}
\newcommand{\rind}{\operatorname{r-ind}}
\newcommand{\lind}{\operatorname{r-ind}}
\newcommand{\ind}{\operatorname{ind}}
\newcommand{\coker}{\operatorname{coker}}
\newcommand{\Hom}{\operatorname{Hom}}
\newcommand{\GL}{\operatorname{GL}}
\newcommand{\tr}{\operatorname{tr}}
\newcommand{\Real}{\operatorname{Re}}
\newcommand{\Imag}{\operatorname{Im}}
%=================================================================

 \title[The Rigid Unit Mode spectrum for symmetric frameworks]{The Rigid Unit Mode spectrum for symmetric frameworks}

\author[E. Kastis]{Eleftherios Kastis}
\email{l.kastis@lancaster.ac.uk}
\address{Dept.\ Math.\ Stats.\\ Lancaster University\\
Lancaster LA1 4YF \\U.K. }

\author[D. Kitson]{Derek Kitson}
\email{derek.kitson@mic.ul.ie}
\address{Dept.\ Math.\ Comp. St.\\Mary Immaculate College,  Ireland.}


\begin{abstract} 
We establish several fundamental properties of the Rigid Unit Mode (RUM) spectrum for symmetric frameworks with a discrete abelian symmetry group and arbitrary linear constraints. In particular, we identify a nonempty subset of the RUM spectrum which derives from the joint eigenvalues of generators for the linear part of the symmetry group. These joint eigenvalues give rise to $\chi$-symmetric flexes which span the space of translations for the framework. 
We show that the RUM spectrum is a union of Bohr-Fourier spectra arising from twisted almost-periodic flexes of the framework. We also characterise frameworks for which every almost periodic flex is a translation. 
\end{abstract}

\maketitle

\section{Introduction}
The Rigid Unit Mode (RUM) model, developed by condensed matter physicists in the 1990's, is used to detect  phase-periodic flexibility in network materials (see the recent survey article \cite{dove} and references therein). 
In recent years, the mathematical underpinnings of the RUM model have been studied from the perspective of operator theory beginning with \cite{owe-pow-crystal, pow-seville, pow-poly}. This theory formalised the connection between three essential tools for analyzing infinitesimal flexes of a crystallographic (i.e.~periodic) bar-and-joint framework $\mathcal{C}$ in $d$-dimensional Euclidean space: the infinite rigidity matrix $R(\mathcal{C})$, the symbol function (or orbit matrices) $\Phi_{\C}(\omega)$ and the RUM spectrum $\Omega(\C)$. Briefly, the rigidity matrix $R(\mathcal{C})$ represents a homogeneous system of linear constraints derived from edge-length preserving motions of the framework and has as its kernel the space of all infinitesimal (i.e.~first-order) flexes of the framework. The symbol function is a matrix-valued function defined on the $d$-torus $\mathbb{T}^d$ with non-zero vectors in the kernel   of $\Phi_\C(\omega)$ corresponding to $\omega$-phase-periodic flexes of the framework $\C$. The RUM spectrum $\Omega(\C)$ is a non-empty subset of $\mathbb{T}^d$ consisting of all phases $\omega$ for which $\omega$-phase-periodic flexes of $\C$ exist.  
In \cite{kkm}, the authors developed a more general theory for symmetric frameworks $\G$ with a discrete abelian symmetry group $\Gamma$ and arbitrary linear constraints. Here the rigidity matrix is replaced by the {\em coboundary matrix} $C(\G)$ which in turn gives rise to a bounded operator $\tilde{C}(\G)$, referred to in this article as the {\em gain framework operator}. The role of the $d$-torus $\mathbb{T}^d$ is played by the dual group $\hat{\Gamma}$ (note that for periodic frameworks $\mathbb{T}^d$ is the Pontryagin dual of the symmetry group $\mathbb{Z}^d$) and phase-periodic infinitesimal flexes are replaced with  {\em $\chi$-symmetric flexes}, where $\chi\in\hat{\Gamma}$. Adapting the well-known result that intertwiners for the bilateral shift on $\ell^2(\bZ)$
are unitarily equivalent to multiplication operators on $L^2(\bT)$, it was shown in \cite{kkm} that the multiplication operator $M_{\Phi_{\G}}$ associated to the symbol function $\Phi_\G$ is related to the gain framework operator $\tilde{C}(\G)$ by a factorisation,
\[M_{\Phi_{\G}}= F_{Y^{E_0}}\circ \tilde{C}(\G) \circ F^{-1}_{X^{V_0}} \circ T_{\tilde{\tau}} ,\]
where $F_{X^{V_0}}$ and $F_{Y^{E_0}}$ are  Fourier transforms  and $T_{\tilde{\tau}}$ is an isometry referred to as a {\em twist} (see Section \ref{s:reps} for the definition). 

In this article, we continue the study of the RUM spectrum for symmetric frameworks but take a more direct approach; working with finite gain frameworks $\G_0$ from the outset rather than the possibly infinite covering frameworks $\G$. 
In Section \ref{s:gain}, we introduce the notion of a {\em gain framework} $\G_0$, its associated {\em gain framework operator} $\tilde{C}(\G_0)$ and {\em orbit matrices} $O_{\G_0}(\chi)$, and the {\em RUM spectrum} $\Omega(\G_0)$. 
We show that the RUM spectrum $\Omega(\G_0)$ always contains a non-empty set $\Omega_{js}(\tau)$ which is inherited from the symmetry group and completely independent of the linear constraints governing the gain framework $\G_0$. The elements of $\Omega_{js}(\tau)$  can be computed from the joint eigenvalues of any generating $n$-tuple for the gain framework's associated point group $d\tau(\Gamma)$ (defined in Section \ref{s:reps}). For this reason we refer to points in $\Omega_{js}(\tau)$ as {\em  joint spectral points}. We show that the joint spectral points for a gain framework $\G_0$ give rise to infinitesimal flexes which generate the space of all translational motions of the covering framework $\G$ (Theorems \ref{t:eigen} and \ref{t:transpace}). 

A second objective of this article is to develop a theory of {\em almost periodic flexibility} for gain frameworks which extends the existing theory for periodic bar-and-joint frameworks developed in \cite{bkp}. 
In Section \ref{s:ap}, we review some necessary theory of scalar-valued almost periodic functions and the extension of this theory to vector-valued functions. For further details on the theory of almost periodic functions on locally compact abelian groups we refer the reader to \cite{bks,rud, shu}.  We introduce the notion of a {\em twisted} almost periodic function and prove a key lemma on intertwining operators (Lemma \ref{L3}).
We then apply this theory to show that the RUM spectrum of a gain framework can be expressed as the union of the Bohr-Fourier spectra of an associated collection of almost periodic functions (Theorem \ref{t:bohr}).
We also characterise gain frameworks which admit only translational almost periodic flexes in terms of the RUM spectrum (Theorem \ref{t:apr}). 

A gain framework captures many of the essential features of a symmetric system of linear constraints and this is demonstrated in Sections \ref{s:applications} and \ref{s:FurtherExamples}.
In Section \ref{s:applications}, the process of constructing a gain framework from a symmetric bar-and-joint framework is explained and  examples of  finite $3$-dimensional bar-and-joint frameworks  with $C_{3h}$-symmetry are presented. 
We demonstrate how to compute the RUM spectrum $\Omega(\G_0)$ and joint spectral points $\Omega_{js}(\tau)$ and how to construct  $\chi$-symmetric flexes for a gain framework $\G_0$ with respect to different choices of linear constraints.
In Section \ref{s:FurtherExamples}, we demonstrate the versatility of the theory with a novel example of an infinite bar-and-joint framework which is not discrete and has irrational rotational symmetry.   We also present an example of an infinite bar-and-joint framework with a finite dimensional space of twisted almost periodic flexes.

%%%%%%%%%%%%%%%%%%%%%%%%%%%%%%%%%%%%%%%%%%%%%%%

\subsection{Historical background}
The origins  of {\em graph rigidity} (also known as {\em combinatorial rigidity}, {\em geometric rigidity}, or {\em structural rigidity}) trace back to   Cauchy's rigidity theorem  (1813)  which, stated informally, says that every convex polyhedron in 3-dimensional Euclidean space is a (continuously) rigid structure (\cite{cau}). 
The various proofs of this famous result demonstrate a striking interplay between geometry and combinatorics which is characteristic of the field today (see for example \cite[Ch.~14]{aig-zie}).
Dehn's proof of Cauchy's rigidity theorem for simplicial polyhedra (\cite{dehn}) and Alexandrov's treatise on convex polytopes (\cite{alex}) developed the now standard methods of {\em infinitesimal rigidity} (see also \cite{asi-rot,asi-rot2,glu}).   
The first examples of flexible polyhedra are due to R.~Bricard (\cite{bri}). However, for these polyhedra to flex the faces of the polyhedra are required to pass through each other and so these examples are not constructible in the real world.
R.~Connelly constructed the first example of a flexible polyhedron which is physically realisable (\cite{con}). This gave rise to the Bellow's Conjecture (since proved) that a continuous flex of a polyhedron preserves volume. The construction of new examples of flexible polyhedra is an ongoing endeavour in mathematics and structural engineering (\cite{ltg}).

In 1864, J.~C.~Maxwell obtained another seminal result in graph rigidity. Suppose we build a $3$-dimensional structure by connecting bars of fixed lengths together where the bars are connected end-to-end and the connecting joints are fully rotational. If the resulting structure is rigid then the number of bars and joints will almost always obey a simple counting rule: $b\geq 3j-6$ where $b$ is the number of bars and $j$ is the number of joints (\cite{max}). The analogous statement for $2$-dimensional structures requires the counting rule $b\geq 2j-3$. What is noteworthy here is that Maxwell's criteria are purely combinatorial, they do not depend on the geometry of the framework. 
In 1923, H.~Pollaczek-Geiringer discovered a counting rule which, together with Maxwell's criteria, provides a complete combinatorial characterisation of (generic) minimal rigidity for bar-and-joint frameworks in the Euclidean plane. 
Roughly speaking, the theorem says that for almost all 2-dimensional bar-and-joint frameworks, minimal rigidity is satisfied if and only if $b=2j-3$ and each subframework obtained by throwing away some of the bars and joints satisfies $b'\leq2j'-3$ where $b'$ is the number of bars and $j'$ is the number of joints in the subframework (\cite{pol}).  This landmark result was 
not widely-known at first but eventually gained prominence when it was rediscovered independently in 1970 by G.~Laman (\cite{Lam}).

In general, the rigidity of a bar-and-joint framework can be tested by computing the rank of a matrix, known as the {\em rigidity matrix}, however this is a computationally slow procedure.
Using the Geiringer-Laman theorem, Lov\'{a}sz and Yemini in 1982  
and Crapo  in 1990 successfully obtained new combinatorial characterisations which gave rise to fast polynomial-time algorithms for testing rigidity in the plane (\cite{crapo,lov-yem}). In 1997, Jacobs and Hendrickson introduced the now standard {\em pebble game} algorithm widely used in software packages for rigidity analysis (\cite{jac-hen}). In 2011, Streinu launched KINARI (KINematics And RIgidity) software for analyzing rigidity properties of structures recorded in the Crystallography Open Database (COD) and Protein Data Bank (PDB) (\cite{kinari}).
Despite these advances, an efficient algorithm for testing the rigidity of generic $3$-dimensional bar-and-joint frameworks is currently unavailable. This is generally regarded as the most important and difficult open problem in the field. 

In 1979, J.~C.~Phillips introduced constraint counting arguments analogous to Maxwell's criteria to explain phase transition phenomena in non-crystalline network materials (\cite{phillips}). In this setting, a bond between two atoms can be viewed as a bar and each atom as a joint. 
An atom with $r$ neighbours is attributed 3 degrees of freedom, $\frac{r}{2}$ bond constraints and $2r-3$ angular constraints. If on average each atom has exactly $2.4$ neighbours then a simple counting argument shows the network lies on the boundary between being underconstrained and overconstrained. In the 1990's, experimentalists in materials science were discovering large numbers of displacive phase transitions in aluminosilicates. Giddy et al.~(\cite{gid-et-al}) explained this unexpected phenomonen by observing that the tetrahedral network structures of aluminosilicates have the exact 2.4 averaging property noted by Phillips. This work motivated the development of the Rigid Unit Mode (RUM) model for network structures (see for example \cite{dove,weg}) and the companion program CRUSH for computing the RUM spectrum (\cite{ham-dov-gid-hei}). 

One of the main reasons research in graph rigidity has intensified in the last 25 years is that new technologies have emerged in which there are fundamental problems that are solvable using graph rigidity. For example, in 2002, Eren et al.~applied techniques from graph rigidity to develop an efficient strategy for maintaining formations of mobile autonomous vehicles (\cite{ebam}).  In 2006, Aspnes et al.~used graph rigidity to develop the theoretical foundations of network localisation for wireless sensor networks and to develop strategies for constructing networks for which the network localisation problem is solvable  (\cite{asp}). In 2009, Krick et al.~presented a decentralized gradient control law to stabilize a network of autonomous mobile robots to any  infinitesimally rigid target formation (\cite{kbf}). In 2010, Singer and Cucuringu demonstrated how graph rigidity can be adapted to solve matrix completion problems which are applicable to computer vision, data analysis, machine learning and collaborative filtering (\cite{sc}). In 2013, Zelazo et al.~ presented a distributed control law for maintaining a rigid formation of autonomous robots which allows for connection and disconnection of links between pairs of robots (\cite{zelazo}). In recent work, alternative distributed control laws for rigidity maintenance have been derived in the context of $\ell_1$ distance constraints (\cite{bcs}) and by considering subframeworks (\cite{pamg}). 
In addition to these application areas, graph rigidity has inspired novel solutions to deep problems in pure mathematics such as  Kalai's proof of the Lower Bound Theorem (\cite{kal}) and Fedorchuk and  Pak's proof of the Robbin's Conjecture (\cite{fed-pak}).

%%%%%%%%%%%%%%%%%%%%%%%%%%%%%%%%%%%%%%%%%%%%%%%%%%%%

\section{Gain frameworks and the RUM spectrum}
\label{s:gain}
Throughout this section, $X$ and $Y$ denote  finite dimensional complex Hilbert spaces and $\Gamma$ denotes a finitely generated additive discrete abelian group.  The group of affine isometries on $X$ is denoted $\Isom(X)$ and the group of linear isometries on $X$ is denoted $U(X)$. The Banach space of bounded functions $f:\Gamma\to X$ with the supremum norm is denoted $\ell^\infty(\Gamma,X)$.
The set of linear transformations from $X$ to $Y$ is denoted $B(X,Y)$.  The Banach algebra of bounded linear transformations on a Banach space $Z$ with the operator norm is denoted $B(Z)$. The identity element in $B(Z)$ is denoted $I$.

%%%%%%%%%%%%%%%%%%%%%%%%%%%%%%%%%%%%%%%%%%%%%%%%%%%%

\subsection{Group representations}
\label{s:reps}
Let $\tau:\Gamma\to \Isom(X)$  be a group homomorphism.
For each $\gamma\in\Gamma$, let $d\tau(\gamma)$ denote the linear part of the affine isometry $\tau(\gamma)$.  %We will require the following fact. 

\begin{Lemma}
\label{l:dt}
$d\tau:\Gamma\to U(X)$ is a unitary representation.   
\end{Lemma}

\proof
It suffices to show that $d\tau$ is a group homomorphism. Let $\gamma_1,\gamma_2\in\Gamma$ and let $x\in X$. Then,
\begin{eqnarray*}
\tau(\gamma_1\gamma_2)x
&=& \tau(\gamma_1)\tau(\gamma_2)x \\
&=& d\tau(\gamma_1)(\tau(\gamma_2)x)+\tau(\gamma_1)0\\
&=& d\tau(\gamma_1)(d\tau(\gamma_2)x+\tau(\gamma_2)0)+\tau(\gamma_1)0 \\
%&=&d\tau(\gamma_1)d\tau(\gamma_2)x+d\tau(\gamma_1)(\tau(\gamma_2)0)+\tau(\gamma_1)0
&=& d\tau(\gamma_1)d\tau(\gamma_2)x+\tau(\gamma_1\gamma_2)0  
\end{eqnarray*}
Thus, $d\tau(\gamma_1\gamma_2) = d\tau(\gamma_1)d\tau(\gamma_2)$.
\endproof

Let $S$ be a finite set. Denote by $X^S$ the  direct sum Hilbert space with inner product $\langle x,y\rangle = \sum_{s\in S}\langle x_s,y_s\rangle$ and by $\tilde{\tau}:\Gamma\to U(X^S)$ the unitary representation, 
%for each $x=(x_{s})_{s\in S}\in X^{S}$, 
$$\tilde{\tau}(\gamma)(x) = (d\tau(\gamma)x_{s})_{s\in S},\quad \forall\, x=(x_{s})_{s\in S}\in X^{S}.$$
%It follows from Lemma \ref{l:dt} that $\tilde{\tau}$ is a unitary representation of $\Gamma$. 
Associated with  $\tilde{\tau}$ is an isometric isomorphism $T_{\tilde{\tau}}\in B(\ell^\infty(\Gamma,X^{S}))$ where, for each $f\in \ell^\infty(\Gamma,X^{S})$,
\[T_{\tilde{\tau}}(f)(\gamma) = \tilde{\tau}(\gamma)f(\gamma).\]
Define a pair of representations $\pi_{X^S}:\Gamma\to B(\ell^\infty(\Gamma,X^S))$   
where,
\[(\pi_{X^S}(\gamma)f)(\gamma')= f(\gamma'-\gamma),\quad \forall\,f\in \ell^\infty(\Gamma,X^S),\]
and  $\pi_{X^{S},\tilde{\tau}}:\Gamma\to B(\ell^\infty(\Gamma,X^{S}))$ where,
\[\pi_{X^{S},\tilde{\tau}}(\gamma) = T_{\tilde{\tau}}\circ \pi_{X^{S}}(\gamma)\circ T_{\tilde{\tau}}^{-1}.\]

\begin{Remark}
   The abelian group $\Gamma$ is assumed throughout to be finitely generated 
 and discrete. It is automatic that the group $\tau(\Gamma)$ is also finitely generated. However, we do not require the group $\tau(\Gamma)$ to be discrete. See Section \ref{s:infinityh} for an example.
\end{Remark}

%%%%%%%%%%%%%%%%%%%%%%%%%%%%%%%%%%%%%%%%%%%%%%%%%%%%

\subsection{Joint spectral points}
Let $T=(T_1,\ldots,T_n)\in B(X)^n$ be an $n$-tuple of pairwise commuting linear transformations on $X$. 
A {\em joint eigenvalue} for $T$ is a point $\lambda=(\lambda_1,\ldots,\lambda_n)\in\mathbb{C}^n$ such that 
there exists a non-zero vector $a\in X$ with $T_ja=\lambda_ja$ for each $j=1,\ldots,n$. The vector $a$ is called a {\em joint eigenvector} for $\lambda$. 
The set of all joint eigenvalues for $T$ is denoted $\sigma(T)$ and referred to as the {\em joint spectrum} of the   $n$-tuple  $T=(T_1,\ldots,T_n)$. The spectrum of a single linear transformation $T\in B(X)$ is similarly denoted $\sigma(T)$. 

The following facts are well-known. 

\begin{Lemma}
\label{l:js}
    If $T=(T_1,\ldots,T_n)\in B(X)^n$ is an $n$-tuple of pairwise commuting linear transformations on  $X$ then:
    \begin{enumerate}[(i)]
        \item $\sigma(T)\not=\emptyset$.
        \item $\sigma(T)\subseteq \sigma(T_1)\times\cdots\times \sigma(T_n)$.
        \item $\sigma(p(T))=p(\sigma(T))$ for all complex polynomial mappings $p:\mathbb{C}^n\to\mathbb{C}^m$.  
    \end{enumerate}
\end{Lemma}

\proof
$(i)$  Since $X$ is finite dimensional there exists an eigenvalue $\lambda_1\in\sigma(T_1)$. Suppose there exists a joint eigenvalue $(\lambda_1,\ldots,\lambda_{m-1})\in \sigma(T_1,\ldots,T_{m-1})$ for some $2\leq m<n$.  
Define $Y_{m-1}:=\cap_{j=1}^{m-1} \ker(T_j-\lambda_jI)$ and note that $Y_{m-1}$ is a non-zero  subspace of $X$.
Moreover, the commutativity of the operators $T_1,\ldots,T_m$ guarantees that $Y_{m-1}$ is an invariant subspace for $T_m$. Thus there exists an eigenvalue $\lambda_m\in \sigma(T_m)$ with eigenvector in $Y_{m-1}$.
It follows that $Y_{m}:=\cap_{j=1}^{m} \ker(T_j-\lambda_jI)$ is non-zero and so $(\lambda_1,\ldots,\lambda_{m})$ is a joint eigenvalue for $(T_1,\ldots,T_{m})$. The result follows. Statement $(ii)$ is clear. For statement $(iii)$ see \cite[Theorem I.2.20]{muller}.
\endproof

Let $\mathcal{A}$ be the commutative unital Banach algebra generated by the pairwise commuting operators $T_1,\ldots,T_n$ in $B(X)$. 
The set $\mathcal{M}_\mathcal{A}$ of all non-zero multiplicative linear functionals $\psi:\mathcal{A}\to\mathbb{C}$ is known as the {\em maximal ideal space} of $\mathcal{A}$.
The joint spectrum of $T$ satisfies,
$$\sigma(T) = \{(\psi(T_1),\ldots,\psi(T_n)):\psi\in\mathcal{M}_\mathcal{A}\}$$
(see \cite[Section I.2]{muller} for example). Given a joint eigenvalue $\lambda\in \sigma(T)$, denote by $\psi_\lambda$ the unique element of $\mathcal{M}_\mathcal{A}$ such that 
$\lambda=(\psi_\lambda(T_1),\ldots,\psi_\lambda(T_n))$.

\begin{Lemma}\label{lem:char}
Let $\tau:\Gamma\to \Isom(X)$  be a group homomorphism and let
$d\tau(\gamma_1),\ldots,d\tau(\gamma_n)$ be a generating set for the group $d\tau(\Gamma)$.
Let $T=(d\tau(\gamma_1),\ldots,d\tau(\gamma_n))\in B(X)^n$
and let $\lambda\in\sigma(T)$.
Then the composition, $$%\chi_\lambda:\Gamma \to\mathbb{T},\qquad 
\chi_\lambda := \psi_\lambda\circ d\tau,$$
lies in the dual group $\hat{\Gamma}$.
\end{Lemma}

\proof
By Lemma \ref{l:dt}, $d\tau$ is a unitary representation.
Since $\psi_\lambda$ is non-zero, bounded and multiplicative, its restriction to the group $d\tau(\Gamma)$ has codomain $\mathbb{T}$.
Thus, the composition $\psi_\lambda\circ d \tau:\Gamma\to\mathbb{T}$ is a group homomorphism and hence a character in  $\hat{\Gamma}$.
\endproof

The conjugate of the character $\chi_\lambda$ in Lemma \ref{lem:char} will be referred to as a {\em  joint spectral point} for the group homomorphism $\tau$ and the generating  $n$-tuple $T=(d\tau(\gamma_1),\ldots,d\tau(\gamma_n))$. The set of all joint spectral points for $\tau$ and  $T$ will be denoted $\Omega_{js}(\tau,T)$,
$$\Omega_{js}(\tau,T)=\{\bar{\chi}_{\lambda}\in\hat{\Gamma}: \lambda\in\sigma(T)\}.$$

\begin{Lemma}
    The set $\Omega_{js}(\tau,T)$  is independent of the choice of generating $n$-tuple $T$.
\end{Lemma}

\proof
%Let $\Gamma_0=\{\gamma_1,\ldots,\gamma_n\}$ and $\Gamma_0'=\{\gamma_1'\ldots,\gamma_m'\}$ be two generating sets for $\Gamma$. 
Let $T=(d\tau(\gamma_1),\ldots, d\tau(\gamma_n))$ and $T'=(d\tau(\gamma_1'),\ldots, d\tau(\gamma_m'))$ be two generating tuples for $d\tau(\Gamma)$.
Let $\lambda'=(\lambda_1',\ldots,\lambda_m')\in \sigma(T')$. We need to show that $\chi_{\lambda'}=\chi_{\lambda}$ for some $\lambda\in\sigma(T)$.
Note that, for each $j=1,\ldots,m$, there exists $j_1,\ldots,j_n\in\mathbb{Z}$ such that $d\tau(\gamma_j') = d\tau(\gamma_1)^{j_1}\cdots d\tau(\gamma_n)^{j_n}$.
Thus $T' = p(T)$ where $p=(p_1,\ldots,p_m)$ is the polynomial mapping with 
$p_j(z_1,\ldots,z_n) = z_1^{j_1}\cdots z_n^{j_n}$ for each $j=1,\ldots,m$.
By the spectral mapping theorem (Lemma \ref{l:js}$(iii)$), $\sigma(T') = \sigma(p(T))=p(\sigma(T))$.
Thus, $\lambda'=p(\lambda)$ for some $\lambda=(\lambda_1,\ldots,\lambda_n)\in\sigma(T)$.
Note that, for each $j=1,\ldots,m$,
$$\chi_{\lambda}(\gamma_j') = \chi_{\lambda}(\gamma_1^{j_1}\cdots\gamma_n^{j_n})
=\chi_{\lambda}(\gamma_1)^{j_1}\cdots\chi_{\lambda}(\gamma_n)^{j_n}
=\lambda_1^{j_1}\cdots\lambda_n^{j_n}=p_j(\lambda)=\lambda_j'=\chi_{\lambda'}(\gamma_j').$$
Thus $\chi_{\lambda'}=\chi_{\lambda}$.
\endproof

In light of the above lemma, we write $\Omega_{js}(\tau) :=\Omega_{js}(\tau,T)$ for some (and hence any) generating $n$-tuple $T=(d\tau(\gamma_1),\ldots,d\tau(\gamma_n))$ and refer to $\Omega_{js}(\tau)$ as the set of  joint spectral points for $\tau$.

%%%%%%%%%%%%%%%%%%%%%%%%%%%%%%%%%%%%%%%%%%%%%%%%%%%%

\subsection{Gain frameworks}
A \emph{$\Gamma$-gain framework} is a tuple
$\G_0=(G_0,m,\varphi,\tau)$ where:
\begin{enumerate}[(i)]
\item $G_0=(V_0,E_0)$ is a directed multigraph,
\item $m:E_0 \rightarrow \Gamma$, $e\mapsto m_e$, is an assignment of group elements to the directed edges of $G_0$,
\item $\varphi=(\varphi_{e})_{e\in E_0}$ is a collection of linear transformations from $X$ to $Y$, and,
\item $\tau:\Gamma\to \Isom(X)$  is a group homomorphism.
\end{enumerate}
The pair $(G_0,m)$ is referred to as the {\em $\Gamma$-gain graph} for $\G_0$.

Given a directed edge  $e\in E_0$, the \emph{source} and \emph{range} of $e$ are denoted by $s(e)$ and $r(e)$, respectively. 
For each $\chi \in\hat{\Gamma}$,  the {\em orbit matrix} $O_{\G_0}(\chi)=(o_{e,v})$ is a matrix with rows indexed by the edge set $E_0$, columns indexed by the vertex set $V_0$ and matrix entries $o_{e,v}\in B(X,Y)$. If $e\in E_0$ is a loop with $v=s(e)=r(e)$ and $w\in V_0$ then the corresponding $(e,w)$-row entry is,
\[o_{e,w}=\left\{\begin{array}{cl}
\varphi_e(I-\chi(m_e)d\tau(m_e)) & \mbox{ if } w=v, \\
0 & \mbox{ otherwise. }\end{array}\right.\] 
Thus, the row of the orbit matrix $O_{\G_0}(\chi)$ which is indexed by the loop edge $e$ has entries, 
\[\kbordermatrix{
 & & & & v & & &\\
 e  & 0&\cdots & 0& \varphi_e(I-\chi(m_e) d\tau(m_e)) & 0 & \cdots &0} \]
  If $e\in E_0$ is not a loop  and $w\in V_0$ then the $(e,w)$-row entry is,
\[o_{e,w}=\left\{\begin{array}{cl}
\varphi_e & \mbox{ if } w=s(e), \\
-\chi(m_e)\varphi_e\circ d\tau(m_e) & \mbox{ if } w=r(e), \\
0 & \mbox{ otherwise. }\end{array}\right.\]
The row of the orbit matrix $O_{\G_0}(\chi)$ which is indexed by the non-loop edge $e$ has entries, 
\[\kbordermatrix{
 & & & & s(e) & & & &  r(e) & & &  \\
 e & 0& \cdots & 0& \varphi_e & 0 & \cdots & 0 &-\chi(m_e)\varphi_e\circ d\tau(m_e) & 0 & \cdots  &0} \]
 Note that each orbit matrix can be regarded as a linear transformation $O_{\G_0}(\chi):X^{V_0}\to Y^{E_0}$. 

  \begin{figure}
        \centering
        \hspace{-3.5cm}
\begin{tikzpicture} 
    % Define the single small vertex labeled v
    \node[circle, draw, fill=black, inner sep=1.5pt, label=below:$v$] (v) at (0,0) {};

    % Draw the three nested directed large loops with arrowheads placed halfway along the loops
    \draw[ postaction={decorate,decoration={markings,mark=at position 0.5 with {\arrow{stealth}}}}]
        (v) to[out=35, in=145, looseness=120] node[midway, above] {$e_1$} (v);
    \draw[postaction={decorate,decoration={markings,mark=at position 0.5 with {\arrow{stealth}}}}]
        (v) to[out=35, in=145, looseness=90] node[midway, above] {$e_2$} (v);
    \draw[postaction={decorate,decoration={markings,mark=at position 0.5 with {\arrow{stealth}}}}] 
        (v) to[out=35, in=145, looseness=40] node[midway, above] {$e_{n-1}$} (v);

    \node[] (d) at (0,1.75) {$\vdots$};
\end{tikzpicture}\hspace{-2cm}
\begin{tikzpicture}
 \centering
    % Define the style for the vertices
    \tikzstyle{vertex}=[circle, draw, fill=black,  inner sep=1.5pt]
    
    % Define the coordinates for the vertices
    \node[vertex, label=below:$v$] (v) at (0,0) {};
    \node[vertex, label=below:$w$] (w) at (3.5,0) {};
    
    % Draw the three parallel directed edges
    \draw[postaction={decorate,decoration={markings,mark=at position 0.5 with {\arrow{stealth}}}}] (v) to[bend left=90,looseness=2 ] node[midway, above] {$e_{n-1}$} (w);
     \draw[postaction={decorate,decoration={markings,mark=at position 0.5 with {\arrow{stealth}}}}] (v) to[bend left=45] node[midway, above] {$e_{2}$} (w);
    \draw[postaction={decorate,decoration={markings,mark=at position 0.5 with {\arrow{stealth}}}}] (v) to node[midway, above] {$e_1$} (w);

    \node[] (d) at (1.75,1.75) {$\vdots$};
\end{tikzpicture}

        \caption{The directed multigraphs   in Examples \ref{Ex:loops} (left) and  \ref{Ex:parallel} (right).}
        \label{fig:loops}
    \end{figure}
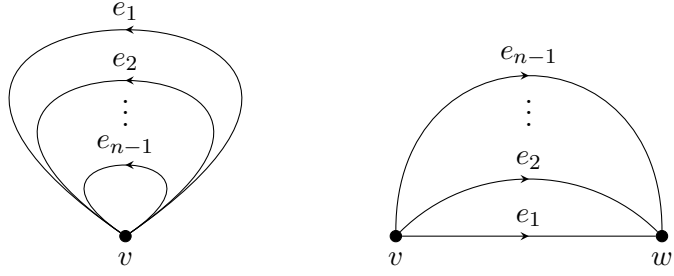

The following two examples present the general form of an orbit matrix for a $\mathbb{Z}_n$-gain framework $\G_0$ with respect to two different multigraphs, an arbitrary collection of linear transformations $\varphi=(\varphi_{v,w})_{vw\in E_0}$ and an unspecified group homomorphism $\tau:\Gamma\to\Isom(X)$. 
Recall that the dual group $\hat{\mathbb{Z}}_n$ is the multiplicative group of characters $\chi_j:\mathbb{Z}_n\to\mathbb{T}$, $\chi_j(k)=\omega^{jk}$, where $\omega=e^{2\pi i/n}$ denotes the $n$th root of unity and $j=0,\ldots,n-1$. 

\begin{Example}
\label{Ex:loops}
Let $\G_0=(G_0,m,\varphi,\tau)$ be a $\mathbb{Z}_n$-gain framework
where $G_0=(V_0,E_0)$ is the directed multigraph   with a single vertex $v$ and $n-1$ loops $e_1,e_2,\ldots,e_{n-1}$ illustrated in  Figure \ref{fig:loops} and $m:E_0\to\mathbb{Z}_n$,  $m(e_k)=k$, for each $k=1,2,\ldots,n-1$. 
Then for each $j$,
$$O_{\G_0}(\chi_j) 
= \kbordermatrix{ 
 & v \\
e_1 &\varphi_{e_1}(I - \omega^j d\tau(1)) \\
e_2 &\varphi_{e_2}(I - \omega^{2j} d\tau(2)) \\
\vdots & \vdots\\
e_{n-1} &\varphi_{e_{n-1}}(I - \omega^{(n-1)j} d\tau(n-1)) }$$
\end{Example}

\begin{Example}
    \label{Ex:parallel}
Let $\G_0=(G_0,m,\varphi,\tau)$ be a $\mathbb{Z}_n$-gain framework where  $G_0=(V_0,E_0)$ is the directed multigraph with two vertices $v$ and $w$ and $n-1$ parallel edges $e_1,e_2,\ldots,e_{n-1}$ illustrated in  Figure \ref{fig:loops}, and $m:E_0\to\mathbb{Z}_n$, $m(e_k)=k$, for each $k=1,2,\ldots,n-1$. 
For each $j$,
$$O_{\G_0}(\chi_j) 
= \kbordermatrix{ 
& v & w \\
e_1 & \varphi_{e_1} & - \omega^j\varphi_{e_1}\circ d\tau(1) \\
e_2 & \varphi_{e_2} & - \omega^{2j}\varphi_{e_2}\circ d\tau(2) \\
\vdots&\vdots&\vdots\\
e_{n-1} & \varphi_{e_{n-1}} & - \omega^{(n-1)j}\varphi_{e_{n-1}}\circ d\tau(n-1)}$$
\end{Example}

 \subsection{The RUM spectrum}
The \emph{Rigid Unit Mode (RUM) spectrum} of a $\Gamma$-gain framework $\G_0$ is denoted $\Omega(\mathcal{G}_0)$ and defined as follows,
\[\Omega(\G_0):=\{\chi\in\hat{\Gamma}:\ker O_{\G_0}(\chi)\neq\{0\}\}.\]

\begin{Theorem}
\label{t:eigen}
    Let $\G_0=(G_0,m,\varphi,\tau)$ be a $\Gamma$-gain framework and let 
%$d\tau(\gamma_1),\ldots,d\tau(\gamma_n)$ be a generating set for the group $d\tau(\Gamma)$. Let $T=(d\tau(\gamma_1),\ldots,d\tau(\gamma_n))\in B(X)^n$.
$T=(d\tau(\gamma_1),\ldots,d\tau(\gamma_n))$ be a generating $n$-tuple for the group $d\tau(\Gamma)$.
    If $\lambda=(\lambda_1,\ldots,\lambda_n)\in \sigma(T)$ is a joint eigenvalue for $T$ with a joint eigenvector $a$ then the column vector
    $[a\,\cdots \,a]^\intercal$ lies in $\ker O_{\G_0}(\bar{\chi}_\lambda)$.
    
    In particular, $\Omega_{js}(\tau)$ is a non-empty subset of the RUM spectrum $ \Omega(\G_0)$.
\end{Theorem}

\proof
Let $e\in E_0$.
Then $d\tau(m_e)=d\tau(\gamma_1)^{j_1}\cdots d\tau(\gamma_n)^{j_n}$ for some $j_1,\ldots,j_n\in\mathbb{Z}$.
Write $\lambda^{m_e}:=\lambda_1^{j_1}\cdots\lambda_n^{j_n}$.
Then, by Lemma \ref{l:dt},
$$d\tau(m_e)a = d\tau(\gamma_1)^{j_1}\cdots d\tau(\gamma_n)^{j_n}a
= \lambda^{m_e}a.$$
Hence,
$\bar{\chi}_\lambda(m_e)  d\tau(m_e)a
= \overline{\lambda^{m_e}} \lambda^{m_e}a=a.$
The result now follows from the definition of $O_{\G_0}(\bar{\chi}_\lambda)$.
\endproof

\begin{Corollary}
\label{c:RUM}
Let $\G_0=(G_0,m,\varphi,\tau)$ be a $\Gamma$-gain framework.
Then, $$|\Omega(\G_0)|\geq |\sigma(T)| \geq \max_{\gamma\in\Gamma}\,|\sigma(d\tau(\gamma))|,$$ 
where $T=(d(\gamma_1),\ldots,d\tau(\gamma_n))$ is any generating $n$-tuple for the group $d\tau(\Gamma)$.
\end{Corollary}

\proof
Let $T=(d(\gamma_1),\ldots,d\tau(\gamma_n))$ be a generating $n$-tuple  for $d\tau(\Gamma)$.
Then, by Theorem \ref{t:eigen}, $|\Omega(\G_0)|\geq |\Omega_{js}(\tau)|=|\sigma(T)|$. 
To prove the second inequality, let $\gamma\in\Gamma$. Note that $d\tau(\gamma)=p(d \tau(\gamma_1),\ldots, d\tau(\gamma_n))$ for some polynomial function $p:\mathbb{C}^n\to \mathbb{C}$.
By Lemma \ref{l:js}$(iii)$, $\sigma(d\tau(\gamma))=p(\sigma(T))$ and so    $|\sigma(T)|\geq|\sigma(d\tau(\gamma))|$.
\endproof

%%%%%%%%%%%%%%%%%%%%%%%%%%%%%%%%%%%%%%%%%%%%%%%%%%%%

\subsection{Gain framework operators} 
\label{ss:gainop}
Let $\G_0=(G_0,m,\varphi,\tau)$ be a $\Gamma$-gain framework.
Define a bounded operator $\tilde{C}(\G_0):\ell^\infty(\Gamma,X^{V_0})\to  \ell^\infty(\Gamma,Y^{E_0})$, called the {\em gain framework operator}, by
$$(\tilde{C}(\G_0)f)(\gamma) = (\varphi_{e}\circ d\tau(-\gamma)(f(\gamma)_{s(e)} - f(\gamma+m_e)_{r(e)}))_{e\in E_0}.$$
A {\em bounded infinitesimal flex} of the $\Gamma$-gain framework $\G_0$  is a vector $u\in \ell^\infty(\Gamma,X^{V_0})$ which satisfies $\tilde{C}(\G_0)u=0$. 

The following proposition was proved in \cite[Proposition 3.5]{kkm} and it is included for the sake of completeness. See Section \ref{s:reps} for the definitions of the representations $\pi_{X^{V_0},\tilde{\tau}}$ and $\pi_{Y^{E_0}}$.
Recall that an {\em intertwiner} for a pair of representations $\pi_1:\Gamma\to B(Z_1)$ and $\pi_2:\Gamma\to B(Z_2)$ is a bounded operator $C:Z_1\to Z_2$ such that, 
$$C\circ \pi_1(\gamma) = \pi_2(\gamma)\circ C,\quad\quad \forall\, \gamma\in \Gamma.$$

\begin{Proposition}
\label{p:intertwine}
    $\tilde{C}(\G_0)$ is an intertwiner for the representations $\pi_{X^{V_0},\tilde{\tau}}$ and $\pi_{Y^{E_0}}$.
\end{Proposition}

\proof 
It suffices to show that the operator $\tilde{C}(\G_0)T_{\tilde{\tau}}$ is an intertwiner for the representations $\pi_{X^{V_0}}$ and $\pi_{Y^{E_0}}$. For each $f\in \ell^\infty(\Gamma,X^{V_0})$, 
\begin{align*}
(\tilde{C}(\G_0)T_{\tilde{\tau}} f)(\gamma) 
&= (\varphi_{e}\circ d\tau(-\gamma)((T_{\tilde{\tau}} f(\gamma))_{s(e)} - (T_{\tilde{\tau}} f(\gamma+m_e))_{r(e)}))_{e\in E_0}\\
&= (\varphi_{e}\circ d\tau(-\gamma)(d\tau(\gamma) f(\gamma)_{s(e)} - d\tau(\gamma+m_e) f(\gamma+m_e)_{r(e)}))_{e\in E_0}\\
&= (\varphi_{e}(f(\gamma)_{s(e)} - d\tau(m_e) f(\gamma+m_e)_{r(e)}))_{e\in E_0}
\end{align*}
Hence, 
\begin{align*}
(\pi_{Y^{E_0}}(\gamma')\tilde{C}(\G_0)T_{\tilde{\tau}} f)(\gamma) &= (\tilde{C}(\G_0)T_{\tilde{\tau}} f)(\gamma-\gamma')\\
&= (\varphi_{e}(f(\gamma-\gamma')_{s(e)} - d\tau(m_e) f(\gamma-\gamma'+m_e)_{r(e)}))_{e\in E_0}\\
&=(\tilde{C}(\G_0)T_{\tilde{\tau}} \pi_{X^{V_0}}(\gamma')f)(\gamma)
\end{align*}
% and
% \begin{align*}
% (\tilde{C}(\G_0)T_{\tilde{\tau}} \pi_{X^{V_0}}(\gamma')f)(\gamma) &= (\varphi_{e}((\pi_{X^{V_0}}(\gamma')f(\gamma))_{s(e)} - d\tau(m_e) (\pi_{X^{V_0}}(\gamma')f(\gamma+m_e))_{r(e)}))_{e\in E_0}\\
% &= (\varphi_{e}(f(\gamma-\gamma')_{s(e)} - d\tau(m_e) f(\gamma-\gamma'+m_e)_{r(e)}))_{e\in E_0}    
% \end{align*}
\endproof

For each $\chi\in\hat{\Gamma}$ and each $a\in X^{V_0}$, define $\chi\otimes a\in \ell^\infty(\Gamma,X^{V_0})$ by setting $(\chi\otimes a)(\gamma)=\chi(\gamma)a$ for each $\gamma\in\Gamma$.

\begin{Proposition}
\label{p:orbit}
Let $\chi\in\hat{\Gamma}$ and $a\in X^{V_0}$. 
Then, for each $\gamma\in\Gamma$,
    $$\tilde{C}(\G_0)T_{\tilde{\tau}}(\chi\otimes a)(\gamma) = \chi(\gamma)O_{\G_0}(\chi)a.$$
\end{Proposition}

\proof
Let $f=\chi\otimes a$.
From the proof of Proposition \ref{p:intertwine}, 
\begin{align*}
(\tilde{C}(\G_0)T_{\tilde{\tau}} f)(\gamma) 
%&= (\varphi_{e}(f(\gamma)_{s(e)} - d\tau(m_e) f(\gamma+m_e)_{r(e)}))_{e\in E_0} \\
&=  (\varphi_{e}(\chi(\gamma)a_{s(e)} - d\tau(m_e) \chi(\gamma+m_e)a_{r(e)}))_{e\in E_0} \\
&= \chi(\gamma) \varphi_e(a_{s(e)} - \chi(m_e) d\tau(m_e) a_{r(e)}))_{e\in E_0} \\
&= \chi(\gamma)O_{\G_0}(\chi)a
\end{align*}
\endproof

\begin{Remark}
The gain framework operator $\tilde{C}(\G_0)$ was introduced in a more restricted form in \cite[Section 3]{kkm}, where it was derived from the coboundary operator of a $\Gamma$-symmetric framework. The coboundary operator (\cite{kkp}) generalises the notion of an oriented incidence matrix for a directed graph and the notion of a rigidity matrix for a bar-and-joint framework. In the latter case, the space of bounded infinitesimal flexes of a gain framework $\G_0$ and its covering framework $\G$ are isomorphic. See Section \ref{s:applications} for further details and examples.
\end{Remark}

%%%%%%%%%%%%%%%%%%%%%%%%%%%%%%%%%%%%%%%%%%%%%%%%%%%%

\subsection{$\chi$-symmetric vectors}
\label{chisymvec}
Let $\mathcal{G}_0=(G_0,m,\varphi,\tau)$ be a $\Gamma$-gain framework.
 Fix $\chi\in\hat{\Gamma}$ and  $a=(a_v)_{v\in V_0}\in X^{V_0}$.
 Define 
$z(\chi,a)=T_{\tilde{\tau}}(\chi\otimes a)\in\ell^\infty(\Gamma, X^{V_0})$.
Note that, for each $\gamma\in\Gamma$,
\begin{eqnarray}
\label{eqn:chi-sym}
z(\chi,a)(\gamma) = T_{\tilde{\tau}}(\chi\otimes a)(\gamma)= (\chi(\gamma)d\tau(\gamma)a_{v})_{v\in V_0}.    
\end{eqnarray}
We refer to $z(\chi,a)$ as a {\em  $\chi$-symmetric vector} for $\G_0$.
A {\em $\chi$-symmetric flex} of $\G_0$ is a $\chi$-symmetric vector  for $\G_0$  which is also a bounded infinitesimal flex of $\G_0$.
    
\begin{Lemma}
\label{l:kernel}
Let $\mathcal{G}_0=(G_0,m,\varphi,\tau)$ be a $\Gamma$-gain framework.
Let $\chi\in\hat{\Gamma}$ and  $a\in X^{V_0}$, $a\not=0$. 
The following statements are equivalent:
\begin{enumerate}[(i)]
    \item $z(\chi,a)$ is a $\chi$-symmetric flex of $\G_0$.
    \item $\chi\in \Omega(\G_0)$ and $a\in \ker O_{\G_0}(\chi)$.
\end{enumerate}
\end{Lemma}

\proof
Apply Proposition \ref{p:orbit}.
\endproof

\begin{Proposition}
\label{l:li}
Let $Z=\{z(\chi_k,a_k)\,:\,k=1,\dots,n\}$ be a set of non-zero $\chi$-symmetric vectors for a $\Gamma$-gain framework $\mathcal{G}_0=(G_0,m,\varphi,\tau)$  with distinct $\chi_k\,,\,k=1,\dots,n$.
Then $Z$ is a linearly independent set.
\end{Proposition}

    \proof
 Suppose that there exist $\lambda_k\in\bC\,,\,k=1,\dots,n$, such that
$$ \sum_{k=1}^n\lambda_k z(\chi_k,a_k)=0$$
or equivalently,
$$ \sum_{k=1}^n\lambda_k T_{\tilde{\tau}}(\chi_k\otimes a_k)(\gamma)=0, \text{ for every }\gamma\in \Gamma.$$
Since $T_{\tilde{\tau}}$ is an isomorphism of $\ell^\infty(\Gamma,X^{V_0})$, it follows that 
 $$ \sum_{k=1}^n\lambda_k \chi_k(\gamma) a_k=0, \text{ for every }\gamma\in \Gamma.$$
Hence the proof follows from the linear independence of the set of characters $\{\chi_k\,:\, k=1\,\dots,n\}$.
    \endproof

\begin{Corollary}\label{c:infdim}
   Let $\G_0=(G_0,m,\varphi,\tau)$ be a $\Gamma$-gain framework.
   If $\Omega(\G_0)$ is an infinite set then the space of bounded infinitesimal flexes of $\G_0$ is infinite dimensional.
\end{Corollary}
\proof
Apply Lemma \ref{l:li}.
\endproof

\begin{Remark}
    It is unknown to the authors if the reverse implication in Corollary \ref{c:infdim} is true. A related result that is known is that the geometric flex spectrum of a periodic bar-and-joint framework in Euclidean space is infinite whenever the space of infinitesimal flexes is infinite dimensional \cite[Theorem 3.7]{kas-pow}. The proof of this result uses spectral synthesis techniques and it would be interesting to know if it could be proved using more direct methods.  
\end{Remark}

\subsection{Joint spectral $\chi$-symmetric vectors}
Let $\mathcal{G}_0=(G_0,m,\varphi,\tau)$ be a $\Gamma$-gain framework and let $T=(d\tau(\gamma_1),\ldots,d\tau(\gamma_n))\in B(X)^n$ be a  generating $n$-tuple for the group $d\tau(\Gamma)$.
A $\chi$-symmetric vector for $\G_0$ of the form $z(\bar{\chi}_\lambda,[a\cdots a]^\intercal)$, where  $\lambda\in \sigma(T)$ is a joint eigenvalue for $T$ with a joint eigenvector $a\in X$, is referred to as a {\em joint  spectral $\chi$-symmetric vector} for $\G_0$.

\begin{Lemma}
\label{l:trivial}
Let $z(\bar{\chi}_\lambda,[a \,\cdots\, a]^\intercal)$ be a joint spectral $\chi$-symmetric vector for a $\Gamma$-gain framework $\mathcal{G}_0=(G_0,m,\varphi,\tau)$. 
Then,
\begin{enumerate}[(i)]
\item $z(\bar{\chi}_\lambda,[a \,\cdots\, a]^\intercal)$ is a $\chi$-symmetric flex of $\G_0$,  and,
\item for each $\gamma\in\Gamma$, $z(\bar{\chi}_\lambda,[a \,\cdots\, a]^\intercal)(\gamma) = [a \,\cdots\, a]^\intercal$.
\end{enumerate}
\end{Lemma}

\proof
Statement $(i)$ follows from Theorem \ref{t:eigen} and Lemma \ref{l:kernel}. To prove statement $(ii)$, let $T=(d\tau(\gamma_1),\ldots,d\tau(\gamma_n))$ be a  generating $n$-tuple for the group $d\tau(\Gamma)$ such that $\lambda\in \sigma(T)$ and $a$ is a joint eigenvector for $\lambda$.
For each $\gamma\in\Gamma$, there exist $j_1,\ldots,j_n\in\mathbb{Z}$
with $d\tau(\gamma)=d\tau(\gamma_1)^{j_1}\cdots d\tau(\gamma_n)^{j_n}$.
Write $\lambda^{\gamma}:=\lambda_1^{j_1}\cdots \lambda_n^{j_n}$.
Then, $$\bar{\chi}_\lambda(\gamma)d\tau(\gamma)a=\bar{\lambda}^\gamma\lambda^\gamma a=a.$$
Thus the result follows from Equation \ref{eqn:chi-sym}.
\endproof
 
A vector $f\in\ell^\infty(\Gamma,X^{V_0})$ is referred to as a {\em translation} of $\G_0$ if there exists $a\in X$ such that $f(\gamma) = [a \,\cdots\, a]^\intercal$ for all $\gamma\in \Gamma$. Note that every translation of $\G_0$ is a bounded infinitesimal flex of $\G_0$.
The vector space of all translations of $\mathcal{G}_0$ is denoted $Z_{\mathcal{G}_0}$ and referred to as the {\em translation space} for $\mathcal{G}_0$.

\begin{Theorem}\label{t:transpace}
    Let $\mathcal{G}_0=(G_0,m,\varphi,\tau)$ be a $\Gamma$-gain framework. The linear span of the joint spectral $\chi$-symmetric vectors for $\G_0$ is the translation space $Z_{\G_0}$. 
\end{Theorem}

\proof
 Let $T=(d\tau(\gamma_1),\ldots,d\tau(\gamma_n))\in B(X)^n$ be a generating $n$-tuple for $d\tau(\Gamma)$ and let $d$ be the dimension of the linear space $X$.
  By Lemma \ref{l:js}$(i)$, there exists a joint eigenvalue $\lambda_1\in \bC^n$ for $T$ with joint eigenvector $a_1\in X$.  Since $T$ is an $n$-tuple of pairwise commuting unitaries, it follows that the spanning space $V_1=\operatorname{span}\{a_1\}$ is a reducing subspace of $X$ for each $d\tau(\gamma_j)$. Let $V_1^\perp$ denote the orthogonal complement of $V_1$. If $V_1^\perp$ is not the trivial subspace, then the restriction of $T$ to $V_1^\perp$ is again an $n$-tuple of pairwise commuting unitaries and so Lemma \ref{l:js}$(i)$ is again applicable. In particular, there exists a joint eigenvalue $\lambda_2\in \bC^n$ for $T$ with  eigenvector $a_2\in V_1^\perp$. Working recursively, we obtain a sequence of (not necessarily distinct) joint eigenvalues $(\lambda_k)_{k=1}^d$ and associated joint eigenvectors $(a_k)_{k=1}^d$ that span $X$.  By Lemma \ref{l:trivial}, each joint spectral $\chi$-symmetric vector $z(\bar{\chi}_{\lambda_k},[a_k \,\cdots\, a_k]^\intercal)$ is a translation of $\mathcal{G}_0$. The joint eigenvectors $(a_k)_{k=1}^d$ are pairwise orthogonal and so the set of joint spectral $\chi$-symmetric vectors, 
  $$\{z(\bar{\chi}_{\lambda_k},[a_k \,\cdots\, a_k]^\intercal)\,:\, k=1,\dots, d\}$$ 
  is linearly independent. Since $Z_{\G_0}$ has dimension $d$ and contains the above set, the result follows.
\endproof

%%%%%%%%%%%%%%%%%%%%%%%%%%%%%%%%%%%%%%%%%%%%%%%%%%%%

\section{Applications to rigidity and flexibility}
\label{s:applications}
In this section, we demonstrate applications of the theory developed thus far to the analysis of infinitesimal flexes for symmetric bar-and-joint frameworks. For an introduction to rigidity theory for finite symmetric bar-and-joint frameworks in Euclidean space we refer the reader to \cite{sch-whi}. For a treatment of periodic bar-and-joint frameworks in Euclidean space see \cite{owe-pow-crystal}. For an introduction to rigidity theory for bar-and-joint frameworks in finite dimensional normed linear spaces see \cite{kit-pow-2,kit-sch1}.  

%%%%%%%%%%%%%%%%%%%%%%%%%%%%%%%%%%%%%%%%%%%%%%%%%%%%

 \subsection{Symmetric frameworks}
 \label{symfram}
Let $G=(V,E)$ be a simple undirected graph and $\varphi=(\varphi_{v,w})_{vw\in E}$ be a family of linear transformations $\varphi_{v,w}:X\to Y$ between finite dimensional real Hilbert spaces $X$ and $Y$  with the property that 
$\varphi_{v,w}=-\varphi_{w,v}$, for each edge $vw\in E$.
We refer to the pair $(G,\varphi)$ as a {\em  framework} (for $X$ and $Y$).

\begin{Example}
\label{Ex:frameworkA}
Let $p:V\to \mathbb{R}^d$, $p(v)=p_v$, be an assignment of points in $\mathbb{R}^d$ to the vertices of a graph $G=(V,E)$ with $p_v\not=p_w$ for each edge $vw\in E$.
Let $\|\cdot\|$ be a norm on $\mathbb{R}^d$ and suppose that, for each edge $vw\in E$,  the norm $\|\cdot\|$ is smooth at the point $p_v-p_w$.
Then the directional derivatives of the norm $\|\cdot\|$ at $p_v-p_w$ define a linear functional $\varphi_{v,w}:\mathbb{R}^d\to\mathbb{R}$ where,
\begin{eqnarray}
\label{eqn:functional}
    \varphi_{v,w}(u) = \lim_{t\to 0} \frac{1}{t}\left(\|p_v-p_w+tu\|-\|p_v-p_w\|\right).
\end{eqnarray} 

The L\"{o}wner ellipsoid $\E$ for the unit ball $B=\{x\in\mathbb{R}^d:\|x\|\leq1\}$ is the unique ellipsoid of minimal volume which contains $B$ (see \cite[\S3.3]{thom}).
Let $X_\E=(\mathbb{R}^d, \|\cdot\|_\E)$ be the real Hilbert space with unit ball $\E$. 
 Then the pair $(G,\varphi)$ is a framework (for the real Hilbert spaces $X_\E$ and $\mathbb{R}$).
 We refer to $(G,\varphi)$ as a {\em bar-and-joint framework} in the normed space $(\mathbb{R}^d, \|\cdot\|)$.
 
Consider now a collection $(\alpha_v)_{v\in V}$ of smooth paths $\alpha_{v}:[-1,1]\to\mathbb{R}^d$ such that $\alpha_{v}(0)=p_v$ for each vertex $v\in V$ and $\|\alpha_v(t)-\alpha_w(t)\|=\|p_v-p_w\|$ for each $t\in[-1,1]$ and each edge $vw\in E$. This edge-length preserving smooth motion gives rise to a family of velocity vectors $(\alpha_v'(0))_{v\in V}$ which satisfy the {\em flex condition} $\varphi_{v,w}(\alpha_v'(0)-\alpha_w'(0))=0$ for each $vw\in E$.
 A vector $u=(u_v)_{v\in V}\in (\mathbb{R}^d)^V$ such that $\varphi_{v,w}(u_v-u_w)=0$ for each edge $vw\in E$ is referred to as an {\em infinitesimal} (or {\em first-order}) {\em flex} of the bar-and-joint framework $(G,\varphi)$.

 In the case of bar-joint frameworks in $d$-dimensional Euclidean space $(\mathbb{R}^d, \|\cdot\|_2)$ it is common practice to work instead with the linear functionals,
 \begin{eqnarray}
\label{eqn:functional2}
\varphi_{v,w}(u) = \|p_v-p_w\|_2\left(\lim_{t\to 0} \frac{1}{t}\left(\|p_v-p_w+tu\|_2-\|p_v-p_w\|_2\right)\right) = (p_v-p_w)\cdot u.
\end{eqnarray}
 The space of infinitesimal flexes for the resulting bar-and-joint framework will be the same regardless of whether we use the linear functionals defined by (\ref{eqn:functional}) or (\ref{eqn:functional2}). 
 %We will use this convention in the examples that follow. 
 Also, note that in this case the associated real Hilbert space $X_\E$ is simply the Euclidean space $(\mathbb{R}^d, \|\cdot\|_2)$.
\end{Example}

Let $\Gamma$ be an additive discrete abelian group. 
A  framework $(G,\varphi)$ (for a pair of real Hilbert spaces $X$ and $Y$) is {\em $\Gamma$-symmetric} if there exist group homomorphisms $\theta:\Gamma\to \Aut(G)$ and $\tau:\Gamma\to \Isom (X)$ such that,
\begin{eqnarray}
\label{e:sym}
    \varphi_{\gamma v, \gamma w} = \varphi_{v,w}\circ d\tau(-\gamma), \quad \forall\,\gamma\in \Gamma.
\end{eqnarray}
Here $\Aut(G)$ denotes the automorphism group of the graph $G$.
For convenience, we write $\gamma v$ instead of $\theta(\gamma)v$ for each $\gamma\in\Gamma$ and each $v\in V$.
We also write $\gamma e$ instead of $(\gamma v)(\gamma w)$ for each edge $e=vw\in E$.
It will be assumed throughout that $\theta$ acts {\em freely} on the vertices and edges of $G$ in the sense that  $\gamma v \neq v$ and $\gamma e \neq e$ for all $\gamma\in\Gamma \backslash \{0\}$ and for all vertices $v\in V$ and edges $e\in E$.
 We denote this $\Gamma$-symmetric framework by $\G=(G,\varphi,\theta,\tau).$ 

\begin{Example}
\label{Ex:frameworkB}
Let $(G,\varphi)$ be a bar-and-joint framework in the normed space $(\mathbb{R}^d, \|\cdot\|)$ arising from an assignment  $p:V\to \mathbb{R}^d$ as in Example \ref{Ex:frameworkA}. Let $\Gamma$ be a discrete abelian group. Suppose there exist group homomorphisms $\theta:\Gamma\to \Aut(G)$ and $\tau:\Gamma\to \Isom (\mathbb{R}^d, \|\cdot\|)$ such that,
$$p_{\gamma v} = \tau(\gamma)p_v,\quad \forall\,v\in V,\,\gamma\in\Gamma.$$
Note that $\Isom (\mathbb{R}^d, \|\cdot\|)$ is a subgroup of $\Isom(X_\E)$ (\cite[Corollary 3.3.4]{thom}).
The linear functionals $\varphi=(\varphi_{v,w})_{vw\in E}$ satisfy Equation \ref{e:sym}
%Note that $\Isom (\mathbb{R}^d, \|\cdot\|)$ is a subgroup of $\Isom(X)$, where $X$ is the real Hilbert space with unit ball given by the L\"{o}wner ellipsoid of $(\mathbb{R}^d, \|\cdot\|)$ (see \cite[Corollary 3.3.4]{thom}).
and so $\mathcal{G}=(G,\varphi,\theta,\tau)$ is a $\Gamma$-symmetric framework.
\end{Example}

%%%%%%%%%%%%%%%%%%%%%%%%%%%%%%%%%%%%%%%%%%%%%%%%%%%%

\subsection{Gain frameworks}
Let $\G=(G, \varphi, \theta, \tau)$ be a $\Gamma$-symmetric framework for a pair of real Hilbert spaces $X$ and $Y$. We construct an associated $\Gamma$-gain framework $\G_0=(G_0, m, \varphi, \tau)$ as follows: 
\begin{enumerate}[(i)]
\item Let $G_0=(V_0, E_0)$ be the quotient graph $G/\sim$ where, for each pair of vertices  $v,w$ in $G$, we write  $v\sim w$ if and only if $v=\gamma w$ for some $\gamma\in\Gamma$.
\item For each vertex orbit $[v]\in V_0$, choose a representative vertex $\tilde{v}\in[v]$ and denote the set of all such representatives by $\tilde{V}_0$.
\item Fix an orientation on the edges of the quotient graph $G_0$. Abusing notation, we denote by $G_0=(V_0,E_0)$ the resulting directed graph. 
\item For each directed edge $[e]=([v],[w])\in E_0$, there exists a unique group element $m_{[e]}\in\Gamma$ such that $\tilde{v}(m_{[e]}\tilde{w})\in [e]$. Let $m:E_0\to \Gamma$ be the assignment $m([e])=m_{[e]}$.
\item Denote by $X_\mathbb{C}$ and $Y_\mathbb{C}$ the complexifications for the real Hilbert spaces $X$ and $Y$.
Let $\varphi=(\varphi_{[e]})_{[e]\in E_0}$ be a collection of linear transformations from $X_\mathbb{C}$ to $Y_\mathbb{C}$ where, for each directed edge $[e]=([v],[w])\in E_0$, we set $\varphi_{[e]}$ to be the complexification of the linear transformation $\varphi_{\tilde{v},m_{[e]}\tilde{w}}$. 
\end{enumerate}
We refer to $\G$ as the {\em covering framework} for the $\Gamma$-gain framework $\G_0$. The graph $G$ is called the {\em covering graph} for the $\Gamma$-gain graph  $(G_0,m)$. The uniqueness of the covering framework for a $\Gamma$-gain framework follows from the assumption that $\theta$ acts freely. 
The mapping $m:[e]\to m_{[e]}$ is referred to as a {\em gain} on $G_0$. Note that since $G$ is a simple graph, the gain $m$ satisfies the following two properties:
\begin{enumerate}[(i)]
 \item  If $[e],[e']$ are 
 parallel edges in $G_0$ with the same orientation then $m_{[e]} \neq m_{[e']}$;
\item $m_{[e]} \neq 0$, for every loop $[e]\in E_0$.
\end{enumerate}

\begin{Remark}
As discussed in Section \ref{symfram}, much of the research in rigidity theory focuses on bar-and-joint frameworks $(G,\varphi)$ with linear functionals $\varphi_{v,w}$ acting on a finite dimensional real Hilbert space $X$. In such cases, given an infinitesimal flex $u$  of the associated gain framework $\G_0$, both the real part $\operatorname{Re}u$ and the imaginary part $\operatorname{Im}u$ are also infinitesimal flexes of $\G_0$.
\end{Remark}

\begin{figure}
\hspace{-3cm}
\begin{tikzpicture}
    \node[circle, draw, fill=black, inner sep=1.5pt, label=below:$v$] (v) at (0,0) {};

    \draw[postaction={decorate,decoration={markings,mark=at position 0.5 with {\arrow{stealth}}}}]
        (v) to[out=35, in=145, looseness=90] node[midway, above] {$(1,1)$} (v);
    \draw[postaction={decorate,decoration={markings,mark=at position 0.5 with {\arrow{stealth}}}}] 
        (v) to[out=35, in=145, looseness=40] node[midway, above] {$(0,1)$} (v);

\end{tikzpicture}
\begin{tikzpicture}
    \node[circle, draw, fill=black, inner sep=1.5pt] (v) at (3,0,0) {};
    \node[circle, draw, fill=black, inner sep=1.5pt] (v) at (0,0,0) {};
    \node[circle, draw, fill=black, inner sep=1.5pt] (v) at (1.5,2,0) {};
    \node[circle, draw, fill=black, inner sep=1.5pt] (v) at (3,0,3) {};
    \node[circle, draw, fill=black, inner sep=1.5pt] (v) at (0,0,3) {};
    \node[circle, draw, fill=black, inner sep=1.5pt] (v) at (1.5,2,3) {};
    % Define vertices of the triangular prism
    \coordinate (A) at (3,0,0);
    \coordinate (B) at (0,0,0);
    \coordinate (C) at (1.5,2,0);
    \coordinate (D) at (3,0,3);
    \coordinate (E) at (0,0,3);
    \coordinate (F) at (1.5,2,3);

    % Calculate centers of the triangular faces
    \coordinate (G) at (barycentric cs:A=1,B=1,C=1);  % Center of ABC
    \coordinate (H) at (barycentric cs:D=1,E=1,F=1);  % Center of DEF

    % Draw base triangle ABC
    \draw[thick] (A) -- (B) -- (C) -- cycle;

    % Draw top triangle DEF
    \draw[thick] (D) -- (E) -- (F) -- cycle;

    % Draw diagnoal edges
    \draw[thick,blue] (A) -- (E);
    \draw[thick,cyan] (B) -- (F);
    \draw[thick,violet] (C) -- (D);
    \draw[thick,violet] (A) -- (F);
    \draw[thick,blue] (B) -- (D);
    \draw[thick,cyan] (C) -- (E);

    % Labels for vertices
    \node[below right] at (A) {$v_{1,2}$};
    \node[below] at (B) {$v_{1,0}$};
    \node[above] at (C) {$v_{1,1}$};
    \node[below right] at (D) {$v_{0,2}$};
    \node[below] at (E) {$v_{0,0}$};
    \node[above left] at (F) {$v_{0,1}$};
\end{tikzpicture}\hspace{-1.5cm}

\caption{The $\Gamma$-gain graph (left) and  covering graph (right) in Example \ref{Ex:2loops}.}
\label{fig:triangularprism}
\end{figure}
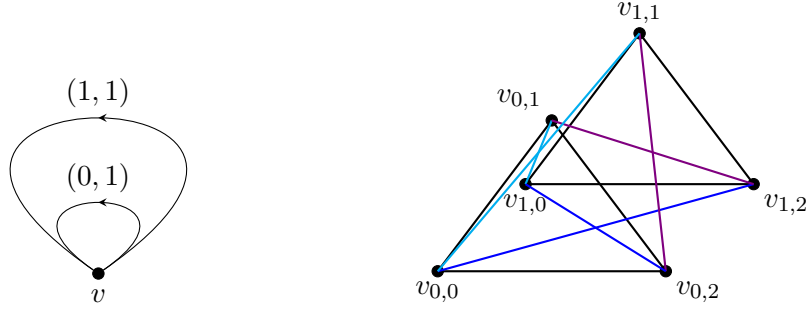

%%%%%%%%%%%%%%%%%%%%%%%%%%%%%%%%%%%%%%%%%%%%%%%%%%%%

\subsection{$C_{3h}$-symmetric frameworks}
In this section we consider frameworks which exhibit $C_{3h}$ symmetry. 
We use here the standard Schoenflies notation whereby $C_{3h}$ is a group generated by a rotation of order $3$ and a reflection in a plane which is orthogonal to the axis of rotation. 

Let $G_0=(V_0,E_0)$ be the directed multigraph with a single vertex $v$ and two loops $e_1,e_2$. Let $\Gamma = \mathbb{Z}_2\times\mathbb{Z}_3$ and define $m:E_0\to\Gamma$ by setting  $m(e_1)=(0,1)$ and $m(e_2)=(1,1)$. %Then the pair $(G_0,m)$ is a $\Gamma$-gain graph. 
 The $\Gamma$-gain graph $(G_0, m)$ and its covering graph are illustrated in Figure \ref{fig:triangularprism}.
   Note that the covering graph $G=(V,E)$ is the octahedral graph with vertex set $V=V_0\times\Gamma$. For convenience, we write $v_{0,0} := (v,(0,0))$ and $v_{l,m}:=\theta(l,m)v_{0,0}=(v,(l,m))$, for each $(l,m)\in \Gamma$. 
   
 Define a group homomorphism $\tau:\Gamma\to \Isom(\mathbb{C}^3)$ by setting $\tau(1,0)$ to be the complexification of orthogonal reflection in the $xy$-plane and $\tau(0,1)$ to be the complexification of clockwise rotation by $\frac{2\pi}{3}$ about the $z$-axis,
$$\tau(1,0) = \begin{bmatrix}
1 &0&0\\
0&1&0\\
0&0&-1
\end{bmatrix},
\qquad
\tau(0,1) = \begin{bmatrix}
-1/2&\sqrt{3}/2 & 0\\
-\sqrt{3}/2&-1/2 & 0\\
0&0&1
\end{bmatrix}.$$ 
Note that $d\tau = \tau$.

Recall that the dual group $\hat{\Gamma}$ is the multiplicative group of characters $\chi_{j,k}:\mathbb{Z}_2\times\mathbb{Z}_3\to\mathbb{T}$, $\chi_{j,k}(l,m)=(-1)^{jl}\eta^{km}$, where $\eta=e^{2\pi i/3}$, $j\in\mathbb{Z}_2$ and $k\in \mathbb{Z}_3$. % and $(l,m)\in \Gamma$.

\begin{Example}
\label{Ex:2loops}
Let $\G_0=(G_0,m,\varphi,\tau)$ be a $\Gamma$-gain framework where  $\varphi=(\varphi_{e_1},\varphi_{e_2})$ is an arbitrary pair of linear transformations from $\mathbb{C}^3$ to $\mathbb{C}$.
For $\chi_{j,k}\in\hat{\Gamma}$, the corresponding orbit matrix is,
\begin{equation}
\label{Ex:2loopsOrbit}
O_{\G_0}(\chi_{j,k})
=\begin{bmatrix} 
\varphi_{e_1}\circ(I - \eta^k \tau(0,1)) \\
\varphi_{e_2}\circ(I - (-1)^j\eta^{k} \tau(1,1)) 
\end{bmatrix}
\end{equation}
By the rank-nullity theorem, $\ker O_{\G_0}(\chi_{j,k})$ is non-trivial for each character $\chi_{j,k}$ and so $\Omega(\G)=\hat{\Gamma}$. 

Let $T=(\tau(1,0),\tau(0,1))$. Note that $T$ is a generating pair for the group  $d\tau(\Gamma)$ and $\sigma(T) = \{(-1,1),(1,\eta),(1,\bar{\eta})\}$.
In particular, $\lambda_1=(-1,1)$ is a joint eigenvalue for $T$ with joint eigenvector $a_{\lambda_1}=(0,0,1)$, $\lambda_2=(1,\eta)$ is a joint eigenvalue for $T$ with joint eigenvector $a_{\lambda_2}=(-i,1,0)$ and $\lambda_3=(1,\bar{\eta})$ is a joint eigenvalue for $T$ with joint eigenvector $a_{\lambda_3}=(i,1,0)$. 
The corresponding characters satisfy, 
$$\bar{\chi}_{\lambda_1}(l,m) = (-1)^l=\chi_{1,0}(l,m),
\quad \bar{\chi}_{\lambda_2}(l,m) = \bar{\eta}^m=\chi_{0,2}(l,m),
\quad \bar{\chi}_{\lambda_3}(l,m) = \eta^{m}=\chi_{0,1}(l,m).$$ 
 In particular,  the joint spectral points in the RUM spectrum are $\Omega_{js}(\tau)=\{\chi_{1,0},\chi_{0,2},\chi_{0,1}\}$.
 
By Lemma \ref{l:trivial}, the corresponding joint spectral $\chi$-symmetric vectors for $\G_0$ are the translations:
\begin{itemize}
    \item $z(\bar{\chi}_{\lambda_1},a_{\lambda_1})(\gamma)=z(\chi_{1,0},a_{\lambda_1})(\gamma) = a_{\lambda_1}$.
\item $z(\bar{\chi}_{\lambda_2},a_{\lambda_2})(\gamma)=z(\chi_{0,2},a_{\lambda_2})(\gamma)=a_{\lambda_2}$.
\item $z(\bar{\chi}_{\lambda_3},a_{\lambda_3})(\gamma)=z(\chi_{0,1},a_{\lambda_3})(\gamma)=a_{\lambda_3}$.
\end{itemize}
 As shown in Theorem \ref{t:transpace}, the above joint spectral $\chi$-symmetric vectors for $\G_0$ span the $3$-dimensional translation space $Z_{\G_0}$.
\end{Example}

The joint spectral $\chi$-symmetric vectors in the above example are  independent of the choice of linear transformations $\varphi=(\varphi_{e_1},\varphi_{e_2})$. 
The remaining $\chi$-symmetric vectors depend on  $\varphi$ and to illustrate this we describe two contrasting examples of $\Gamma$-symmetric bar-and-joint frameworks below. In both examples, the linear transformations $\varphi$ are derived from norm constraints as described in Example \ref{Ex:frameworkA} using the assignment $p:V\to \mathbb{R}^3$ with,
\begin{align*} 
p(v_{0,0}) &:=(-\sqrt{3},-1,1), & &p(v_{1,0}) = \tau(1,0)p(v_{0,0})=(-\sqrt{3},-1,-1),\\
p(v_{0,1}) &= \tau(0,1)p(v_{0,0})=(0,2,1), && p(v_{1,1}) = \tau(1,1)p(v_{0,0})=(0,2,-1),\\
p(v_{0,2})&=\tau(0,2)p(v_{0,2})=(\sqrt{3},-1,1), && p(v_{1,2})=\tau(1,2)p(v_{0,2})=(\sqrt{3},-1,-1).
\end{align*}
The bar-and-joint framework resembles a triangular prism with two triangular faces connected by three pairs of crossing bars. See the right hand image in Figure \ref{fig:triangularprism} for an illustration. 
The linear transformations $\varphi$ represent standard Euclidean norm constraints in Example \ref{Ex:2loopsA} and cylindrical norm constraints in Example \ref{Ex:2loopsB}.
 
\begin{Example}
\label{Ex:2loopsA}
Let $\G_0=(G_0,m,\varphi,\tau)$ be the $\Gamma$-gain framework with  linear transformations $\varphi=(\varphi_{e_1},\varphi_{e_2})$ given by,
\begin{align*}
\varphi_{e_1}(a) &= (p(v_{0,0})-\tau(0,1)p(v_{0,0}))\cdot a =(-\sqrt{3}, -3, 0)\cdot a,\\
\varphi_{e_2}(a) &= (p(v_{0,0})-\tau(1,1)p(v_{0,0}))\cdot a =(-\sqrt{3}, -3, 2)\cdot a,
\end{align*}
for all $a\in \mathbb{C}^3$. 
Using Equation \ref{Ex:2loopsOrbit}, the orbit matrices for $\G_0$ take the form,
 \begin{eqnarray*}
 O_{\G_0}(\chi_{j,k})  
 &=& \begin{bmatrix} 
 -\sqrt{3}(1+2\eta^{k}) & -3 & 0\\
 -\sqrt{3}(1+2(-1)^j\eta^{k}) & -3 & 2(1+(-1)^j\eta^{k}) \\
 \end{bmatrix}
 \end{eqnarray*}
 Note that $\dim \ker O_{\G_0}(\chi_{j,k}) =1$ for all $(j,k)\in\Gamma$.
 
 The vector $u_{0,0}=(1,-\sqrt{3},0)$ spans the kernel of $O_{\G_0}(\chi_{0,0})$. The $\chi$-symmetric vector $z(\chi_{0,0},u_{0,0})$ describes an anticlockwise rotation of the covering framework about the $z$-axis where, for $l=0,1$,
  $$z_{(l,0)}= \begin{bmatrix} 1\\ -\sqrt{3} \\ 0 \end{bmatrix},\qquad
  z_{(l,1)}= \begin{bmatrix} -2\\ 0 \\ 0 \end{bmatrix},\qquad
  z_{(l,2)}= \begin{bmatrix} 1\\ \sqrt{3} \\ 0 \end{bmatrix}.$$
 
 The vector $u_{1,1}=(i,1, -2\overline{\eta})$ spans the kernel of $O_{\G_0}(\chi_{1,1})$. The components of the $\chi$-symmetric vector $z(\chi_{1,1},u_{1,1})$ are,  
  $$z_{(0,0)}= \begin{bmatrix} i\\ 1 \\ -2\bar{\eta}  \end{bmatrix},\qquad
  z_{(0,1)}= \begin{bmatrix}  i\\ 1 \\ -2  \end{bmatrix},\qquad
  z_{(0,2)}= \begin{bmatrix}  i\\ 1 \\ -2\eta  \end{bmatrix}.$$
$$z_{(1,0)}= \begin{bmatrix} -i\\-1 \\ -2\bar{\eta}  \end{bmatrix},\qquad
  z_{(1,1)}= \begin{bmatrix} -i\\-1 \\ -2  \end{bmatrix},\qquad
  z_{(1,2)}= \begin{bmatrix} -i\\-1 \\ -2\eta  \end{bmatrix}.$$
  The real part of $z(\chi_{1,1},u_{1,1})$ describes a rotation of the covering framework about the $x$-axis. The imaginary part of  $z(\chi_{1,1},u_{1,1})$ describes a rotation of the covering framework about the $y$-axis. 
  
 The vector $u_{1,2}=(i,-1, 2\eta)$ spans the kernel of $O_{\G_0}(\chi_{1,2})$. The components of the $\chi$-symmetric vector $z(\chi_{1,2},u_{1,2})$ are,
 $$z_{(0,0)}= \begin{bmatrix} i\\-1 \\ 2\eta  \end{bmatrix},\qquad
  z_{(0,1)}= \begin{bmatrix} i\\ -1 \\ 2  \end{bmatrix},\qquad
  z_{(0,2)}= \begin{bmatrix} i\\ -1 \\ 2 \bar{\eta}  \end{bmatrix}.$$
$$z_{(1,0)}= \begin{bmatrix} -i\\ 1 \\ 2\eta  \end{bmatrix},\qquad
  z_{(1,1)}= \begin{bmatrix} - i\\ 1 \\ 2  \end{bmatrix},\qquad
  z_{(1,2)}= \begin{bmatrix} - i\\ 1 \\ 2 \bar{\eta}  \end{bmatrix}.$$
 The real part of $z(\chi_{1,2},u_{1,2})$ describes a rotation of the covering framework about the $x$-axis which is the inverse rotation to the real part of $z(\chi_{1,1},u_{1,1})$. The imaginary part of  $z(\chi_{1,2},u_{1,2})$ describes a rotation of the covering framework about the $y$-axis which is identical to the imaginary part of $z(\chi_{1,1},u_{1,1})$. 
\end{Example}

In the following example, $\ell_{2,\infty}^3$ denotes the cylindrical normed space $(\mathbb{R}^3,\|\cdot\|_{2,\infty})$ where,
$$\|(x,y,z)\|_{2,\infty} = \max\,\{\sqrt{x^2+y^2}, |z|\}.$$
\begin{Example}
\label{Ex:2loopsB}
Let $\G_0=(G_0,m,\varphi,\tau)$ be the $\Gamma$-gain framework with linear transformations $\varphi=(\varphi_{e_1},\varphi_{e_2})$ given by,
\begin{align*} 
\varphi_{e_1}(a) &= \pi_{xy}(p(v_{0,0})-\tau(0,1)p(v_{0,0}))\cdot a=(-\sqrt{3}, -3, 0)\cdot a,\\
\varphi_{e_2}(a) &= \pi_{z}(p(v_{0,0})-\tau(1,1)p(v_{0,0}))\cdot a=(0,0,2)\cdot a,
\end{align*}
 for each $a\in \mathbb{C}^3$, where $\pi_{xy}$ and $\pi_z$ are the projections,
 $$\pi_{xy}(a_1,a_2,a_3) = (a_1,a_2,0),\qquad \pi_z(a_1,a_2,a_3)=(0,0,a_3).$$ 
% The linear transformations $\varphi=(\varphi_{e_1},\varphi_{e_2})$ are derived from the cylindrical $\ell^3_{2,\infty}$-distance constraints on the covering framework. 
Note that $\tau(\gamma)$ is an isometry of the cylindrical space $\ell_{2,\infty}^3:=\mathbb{R}^2\oplus_\infty \mathbb{R}$ for each $\gamma\in\Gamma$.

Using Equation \ref{Ex:2loopsOrbit}, the orbit matrices for $\G$ take the form,
 \begin{eqnarray*}
 O_{\G_0}(\chi_{j,k})  
 &=& \begin{bmatrix} 
-\sqrt{3}(1+2\eta^{k}) & -3 & 0 \\
0 & 0 & 2(1+(-1)^j\eta^k) \\
 \end{bmatrix}
 \end{eqnarray*}
 Note that $\dim \ker O_{\G_0}(\chi_{j,k}) =1$ for all $(j,k)\in\Gamma\backslash\{(1,0)\}$, and  $\dim \ker O_{\G_0}(\chi_{1,0}) =2$.
 
 The vector $u_{0,0}=(1,-\sqrt{3},0)$ spans the kernel of $O_{\G_0}(\chi_{0,0})$. The $\chi$-symmetric vector $z(\chi_{0,0},u_{0,0})$ describes an anticlockwise rotation of the covering framework about the $z$-axis where, for $l=0,1$,
  $$z_{(l,0)}= \begin{bmatrix} 1\\ -\sqrt{3} \\ 0 \end{bmatrix},\qquad
  z_{(l,1)}= \begin{bmatrix} -2\\ 0 \\ 0 \end{bmatrix},\qquad
  z_{(l,2)}= \begin{bmatrix} 1\\ \sqrt{3} \\ 0 \end{bmatrix}.$$
 
 The vector $u_{1,1}=(i,1, 0)$ spans the kernel of $O_{\G_0}(\chi_{1,1})$. The components of the $\chi$-symmetric vector $z(\chi_{1,1},u_{1,1})$ are,  
  $$z_{(0,0)}= \begin{bmatrix} i\\ 1 \\ 0  \end{bmatrix},\qquad
  z_{(0,1)}= \begin{bmatrix}  i\\ 1 \\ 0  \end{bmatrix},\qquad
  z_{(0,2)}= \begin{bmatrix}  i\\ 1 \\ 0 \end{bmatrix}.$$
$$z_{(1,0)}= \begin{bmatrix} -i\\-1 \\ 0 \end{bmatrix},\qquad
  z_{(1,1)}= \begin{bmatrix} -i\\-1 \\ 0  \end{bmatrix},\qquad
  z_{(1,2)}= \begin{bmatrix} -i\\-1 \\ 0 \end{bmatrix}.$$
  The real part of $z(\chi_{1,1},u_{1,1})$ describes a shearing motion on the covering framework in which the triangles are translated in opposite directions parallel to the $y$-axis. The imaginary part of  $z(\chi_{1,1},u_{1,1})$ describes an analogous shearing motion parallel to the $x$-axis. 
  
 The vector $u_{1,2}=(i,-1, 0)$ spans the kernel of $O_{\G_0}(\chi_{1,2})$. The components of the $\chi$-symmetric vector $z(\chi_{1,2},u_{1,2})$ are,
 $$z_{(0,0)}= \begin{bmatrix} i\\-1 \\ 0 \end{bmatrix},\qquad
  z_{(0,1)}= \begin{bmatrix} i\\ -1 \\ 0  \end{bmatrix},\qquad
  z_{(0,2)}= \begin{bmatrix} i\\ -1 \\ 0   \end{bmatrix}.$$
$$z_{(1,0)}= \begin{bmatrix} -i\\ 1 \\ 0  \end{bmatrix},\qquad
  z_{(1,1)}= \begin{bmatrix} - i\\ 1 \\ 0 \end{bmatrix},\qquad
  z_{(1,2)}= \begin{bmatrix} - i\\ 1 \\ 0  \end{bmatrix}.$$
 The real part of $z(\chi_{1,2},u_{1,2})$ describes a shearing motion on the covering framework which is the negative of the shearing described by the real part of $z(\chi_{1,1},u_{1,1})$. The imaginary part of  $z(\chi_{1,2},u_{1,2})$ describes a shearing motion on the covering framework which is identical to the imaginary part of $z(\chi_{1,1},u_{1,1})$. 

 The vectors $u_{1,0}=(1,-\sqrt{3}, 0)$ and $a_{\lambda_1}=(0,0,1)$ span the kernel of $O_{\G_0}(\chi_{1,0})$. 
 The components of the $\chi$-symmetric vector $z(\chi_{1,0},u_{1,0})$ are,
 $$z_{(0,0)}= \begin{bmatrix} 1\\ -\sqrt{3} \\ 0 \end{bmatrix},\qquad
  z_{(0,1)}= \begin{bmatrix} -2\\ 0 \\ 0  \end{bmatrix},\qquad
  z_{(0,2)}= \begin{bmatrix} 1\\ \sqrt{3} \\ 0   \end{bmatrix}.$$
$$z_{(1,0)}= \begin{bmatrix} -1\\ \sqrt{3} \\ 0  \end{bmatrix},\qquad
  z_{(1,1)}= \begin{bmatrix} 2\\ 0 \\ 0 \end{bmatrix},\qquad
  z_{(1,2)}= \begin{bmatrix} -1\\ -\sqrt{3} \\ 0  \end{bmatrix}.$$
The $\chi$-symmetric vector $z(\chi_{1,0},u_{1,0})$ describes a twisting motion with one triangle rotating anticlockwise about the $z$-axis and the other triangle rotating clockwise about the $z$-axis.
\end{Example}

%%%%%%%%%%%%%%%%%%%%%%%%%%%%%%%%%%%%%%%%%%%%%%%%%%%%

\section{Twisted almost periodic flexes}
\label{s:ap}
In this section, the class of twisted almost periodic flexes for a gain framework $\G_0$ is introduced and the following two main results are proved: Firstly, it is shown that the RUM spectrum $\Omega(\G_0)$ can be expressed as a union of Bohr-Fourier spectra associated to the twisted almost periodic flexes of $\G_0$ (Theorem \ref{t:bohr}). Secondly, a characterisation is provided for gain frameworks with the property that every twisted almost periodic flex is a translation (Theorem \ref{t:apr}).
For an introduction to the theory of almost periodic functions on locally compact abelian groups we refer the reader to \cite{hew-ros,rud,shu}. 
Given two topological spaces $W$ and $Z$, we denote by $C(W,Z)$ the set of all continuous functions from $W$ to $Z$.

%%%%%%%%%%%%%%%%%%%%%%%%%%%%%%%%%%%%%%%%%%%%%%%%%%%%

\subsection{$\mathbb{C}$-valued almost periodic functions}
Let $\Gamma$ be a countable discrete abelian group and denote by $\hat{\Gamma}$ the dual group of continuous homomorphisms $\chi:\Gamma\to \mathbb{T}$,  where $\mathbb{T}=\{z\in\mathbb{C}:|z|=1\}$ denotes the circle group.
Denote by $\hat{\Gamma}_\delta$ the locally compact Abelian group obtained by endowing the dual group $\hat{\Gamma}$ with the discrete topology. 
We endow $C(\hat{\Gamma}_\delta,\mathbb{T})$  with the compact open topology. 
The {\em Bohr compactification} of $\Gamma$, denoted $\Gamma_B$, is the dual group of $\hat{\Gamma}_\delta$ together with the subspace topology inherited from $C(\hat{\Gamma}_\delta,\mathbb{T})$. 

A trigonometric polynomial on $\Gamma$ is a function $p:\Gamma\to\mathbb{C}$ of the form 
$$p(\gamma)=\sum_{k=1}^{n} a_k \chi_k(\gamma), \quad \text{ where }  a_k\in\mathbb{C},\,\, \chi_k\in \hat{\Gamma} \text{ and } \gamma\in \Gamma.$$
A function $f:\Gamma\rightarrow \mathbb{C}$ is \emph{almost periodic} if,  for every $\varepsilon>0$, there exists a trigonometric polynomial $p_\varepsilon$ such that $\|f-p_\varepsilon\|_\infty<\varepsilon$. 
We denote by $AP(\Gamma,\mathbb{C})$ the set of all complex-valued almost periodic functions on $\Gamma$. One can check that $AP(\Gamma,\mathbb{C})$ is a norm-closed self-adjoint subalgebra of the C$^*$-algebra $\ell^\infty(\Gamma,\mathbb{C})$ of bounded complex-valued functions on $\Gamma$, and hence $AP(\Gamma,\mathbb{C})$ is a commutative C$^*$-algebra. Using the Stone-Weierstrass theorem, it can be shown that $AP(\Gamma,\mathbb{C})$ is isomorphic to the C$^*$-algebra $C(\Gamma_B,\mathbb{C})$ of continuous complex-valued functions on the Bohr compactification $\Gamma_B$.  

\begin{Theorem}\label{t:firstapprox}\cite[Proposition 1.2]{shu}
Let $\Gamma$ be a countable discrete abelian group and let $f:\Gamma\rightarrow \mathbb{C}$. The following are equivalent:
\begin{enumerate}[(i)]
\item $f$ is almost periodic.
\item $f=g\circ i_B$ for some $g\in C(\Gamma_B,\mathbb{C})$, where $i_B$ is the natural embedding of $\Gamma$ into $\Gamma_B$.
\end{enumerate}
\end{Theorem}
Recall that the map $i_B:\Gamma\to\Gamma_B$ is given by the formula $i_B(\gamma)(\chi)=\chi(\gamma)$, where $\gamma\in \Gamma$ and $\chi\in \hat{\Gamma}$. Using Pontryagin duality it can be shown that $i_B(\Gamma)$ is dense in $\Gamma_B$. It follows that the function $g$ in Theorem \ref{t:firstapprox} is uniquely defined, so we also write  $i_B(f)=g$. The map $i_B:AP(\Gamma,\mathbb{C})\to C(\Gamma_B,\mathbb{C})$ is an isomorphism between  C$^*$-algebras, and hence isometric.

Since $\Gamma$ is a countable discrete abelian group, there exists an increasing sequence $(H_n)$ of finite subsets of $\Gamma$ such that
$\Gamma = \cup_n H_n$ and for each $\gamma\in\Gamma$,
\begin{equation}\label{setdiffer}
\lim_{n\to\infty} \frac{|\gamma H_n\cap H_n^c|}{|H_n|} = 0, 
\end{equation}
where $\gamma H_n=\{\gamma+\gamma' : \gamma'\in H_n\}$,  $|S|$ denotes the cardinality of a set $S$ and $S^c$ denotes the complement of $S$ (see \cite[Lemma 18.13]{hew-ros}).
We refer to $(H_n)$ as a \emph{Bohr-Bochner sequence} for $\Gamma$.
Given $f\in AP(\Gamma, \mathbb{C})$ define, 
\[M(f) := \lim_{n\to\infty} \frac{1}{|H_n|}\sum_{\gamma\in H_n} f(\gamma) \]
Then $M$ is a positive linear functional on $AP(\Gamma,\mathbb{C})$ that is independent of the choice of the sequence $(H_n)$. Moreover, the mean $M$ is invariant under translations, meaning that $M(f) = M(D_{\gamma'}(f))$ for all $\gamma'\in\Gamma$ and $f\in AP(\Gamma,\mathbb{C})$, where we define, 
    $$D_{\gamma'}:\ell^\infty(\Gamma,\mathbb{C})\to \ell^\infty(\Gamma,\mathbb{C}), \qquad f(\gamma)\mapsto f(\gamma-\gamma').$$
It can be shown that every positive  linear functional $M'$ on $AP(\Gamma,\mathbb{C})$ that is invariant under translations, is a positive multiple of $M$. Since $M(1)=1$, it follows that 
\[M(f)=\int_{\Gamma_B} i_B(f)\, d\mu,\]
where $\mu$ is the normalized Haar measure on $\Gamma_B$ (see \cite[Lemma 18.9]{hew-ros}). 

The following lemma demonstrates the importance of property \eqref{setdiffer}.

\begin{Lemma}\label{sumchar}
Let $\chi\in \hat{\Gamma}$. Then, 
\begin{equation*}
M(\chi)=    \begin{cases}
1, \text{ if } \chi = 1\\
0, \text{ if } \chi\neq 1
\end{cases}
\end{equation*}
\end{Lemma}
\begin{proof}
If $\chi = 1$, then the result follows by a direct calculation. Suppose there exists $\gamma_0\in \Gamma$, such that $\chi(\gamma_0)\neq 1$. Choose a Bohr-Bochner sequence $(H_n)$ for $\Gamma$. Then,
\begin{align*}
\chi(\gamma_0) M(\chi)&= \lim_{n\to\infty} \frac{1}{|H_n|}\sum_{\gamma\in H_n} \chi(\gamma_0+\gamma)= \lim_{n\to\infty} \frac{1}{|H_n|}\sum_{\gamma\in \gamma_0 H_n} \chi(\gamma)\\
&= \lim_{n\to\infty} \frac{1}{|H_n|}\left(\sum_{\gamma\in H_n} \chi(\gamma)+\sum_{\gamma\in \gamma_0 H_n\cap H_n^c} \chi(\gamma)-\sum_{\gamma\in H_n\cap \gamma_0 H_n^c} \chi(\gamma)\right)
\end{align*}
It follows by property \eqref{setdiffer} that,
\begin{align*}
\lim_{n\to\infty} \bigg| \frac{1}{|H_n|}\sum_{\gamma\in \gamma_0 H_n\cap H_n^c} \chi(\gamma)\bigg|\leq 
\lim_{n\to\infty} \frac{1}{|H_n|}\sum_{\gamma\in \gamma_0 H_n\cap H_n^c} |\chi(\gamma)|= 
\lim_{n\to\infty} \frac{|\gamma_0 H_n\cap H_n^c|}{|H_n|}=0.    
\end{align*}
Since $|H_n\cap \gamma_0 H_n^c|=|(-\gamma_0 H_n)\cap H_n^c|$, the same argument shows that,
\[\lim_{n\to\infty} \bigg| \frac{1}{|H_n|}\sum_{\gamma\in H_n\cap \gamma_0 H_n^c} \chi(\gamma)\bigg|=0,\]
so we obtain, 
\[\chi(\gamma_0) M(\chi)=M(\chi).\]
Since $\chi(\gamma_0)\neq 1$, the result follows.
\end{proof}

%%%%%%%%%%%%%%%%%%%%%%%%%%%%%%%%%%%%%%%%%%%%%%%%%%%%

\subsection{Vector-valued almost periodic functions}
The notion of almost periodicity can be extended to vector-valued functions in the following way: Let $X$ be a finite dimensional complex Hilbert space and denote by $\ell^\infty(\Gamma, X)$ the Banach space of bounded $X$-valued functions on $\Gamma$ with the supremum norm. An $X$-valued trigonometric polynomial on $\Gamma$ is a function $p:\Gamma \to X$ given by 
$$p(\gamma)=\sum_{k=1}^{n} a_k \chi_k(\gamma), \text{ where } \chi_k\in \hat{\Gamma},\, \gamma\in \Gamma \text{ and } a_k\in X.$$
We define the space $AP(\Gamma,X)$ of $X$-valued almost periodic functions as the closure in $\ell^\infty(\Gamma,X)$ of the $X$-valued trigonometric polynomials on $\Gamma$.
We identify $AP(\Gamma,X)$ with the space $AP(\Gamma, \mathbb{C})\otimes X$ using the map, $$ f\otimes a\mapsto \tilde{f}_a : \tilde{f}_a(\gamma)=f(\gamma)a, \quad f\in AP(\Gamma, \mathbb{C}),\quad a\in X,$$
and extending linearly. By abuse of notation, we extend the definition of the mean $M$ to $AP(\Gamma,X)$ by setting,
$$M(f\otimes a):= M(f)a.$$
For $f\in AP(\Gamma, \mathbb{C})$ and $h\in AP(\Gamma, X)$, we write the \emph{mean inner product},
\[   [f,h]:=  M(\bar{f}h)=\lim_{n\to\infty} \frac{1}{|H_n|}\sum_{\gamma\in H_n} \overline{f(\gamma)} h(\gamma)  \]
and, for each $\gamma\in \Gamma$, the \emph{mean convolution formula},
\[(f\ast h)(\gamma)= \lim_{n\to\infty} \frac{1}{|H_n|}\sum_{\gamma'\in H_n} f(\gamma-\gamma')h(\gamma').\]
Note that the mean inner product and the mean convolution formula are independent of the choice of Bohr-Bochner sequence $(H_n)$.
For $\chi\in\hat{\Gamma}$, we define the Fourier coefficient $\hat{h}(\chi)\in X$ of $h$ by,
\[\hat{h}(\chi):=[\chi,h]=\lim_{n\to\infty} \frac{1}{|H_n|}\sum_{\gamma\in H_n} \overline{\chi(\gamma)}h(\gamma).\]
Note that for $\chi'\in \hat{\Gamma}$ and $a\in X$, Lemma \ref{sumchar} yields,
\[ \widehat{(\chi'\otimes a)}(\chi)
= [\chi,\chi'\otimes a]=M((\bar{\chi}\chi')\otimes a)=M(\bar{\chi}\chi')a
=\left\{\begin{array}{ll}
a,& \mbox{ if }\chi'=\chi \\
0,& \mbox{ if }\chi'\not=\chi \\
\end{array}\right..\]
We write $\Lambda(h)=\{\chi\in \hat{\Gamma}: \hat{h}(\chi)\neq 0\}$ for the \emph{Bohr-Fourier spectrum} of $h\in AP(\Gamma, X)$. 

%%%%%%%%%%%%%%%%%%%%%%%%%%%%%%%%%%%%%%%%%%%%%%%%%%%%

\subsection{Bochner-Fejer polynomials}
Next we see that every $X$-valued almost periodic function $h\in AP(\Gamma,X)$ is the uniform limit of a net of $X$-valued trigonometric polynomials $(\sigma_\lambda(h))_{\lambda\in L}$, such that $\Lambda(\sigma_\lambda(h))\subseteq \Lambda(h)$ for each $\lambda$ in a directed set $L$.

Let $\mathcal{U}$ be an open neighborhood basis of the identity in the Bohr compactification $\Gamma_B$ and fix $U\in\mathcal{U}$.  By Urysohn's lemma, there exists a non-negative function $g_U\in C(\Gamma_B,\mathbb{C})$ such that $\supp(g_U)\subseteq U$ and $M(g_U\circ i_B)=\int_{\Gamma_B} g_U\, d\mu=1$. 
By Theorem \ref{t:firstapprox}, the composition $g_U\circ i_B$ lies in $AP(\Gamma,\mathbb{C})$. 
Thus, given $\varepsilon>0$, there exists a real trigonometric polynomial $p_{U,\varepsilon}\in AP(\Gamma,\mathbb{C})$ such that, $$\|g_U - i_B(p_{U,\varepsilon})\|_\infty = \|g_U\circ i_B - p_{U,\varepsilon}\|_\infty
<\varepsilon.$$ Adding $\varepsilon$ and normalizing if needed, we may assume that $p_{U,\varepsilon}\geq 0$ and $M(p_{U,\varepsilon})=1$.

Define the directed set $L=\mathcal{U}\times (0,1)$ with the partial ordering $\prec$ given by,
$$(U_1,\varepsilon_1)\prec(U_2,\varepsilon_2)\text{ when }U_2\subseteq U_1 \text{ and }\varepsilon_2\leq \varepsilon_1.$$ 
Then the net of trigonometric polynomials $(p_\lambda)_{\lambda\in L}$ satisfies the kernel properties:

\begin{enumerate}[(i)]
\item $p_\lambda$ is non-negative;
\item $M(p_\lambda)=1$; and
\item $\lim\limits_{\lambda\in L} \,\left(\sup\{i_B(p_\lambda)(x):x\notin U\}\right)=0$ for each open neighbourhood $U$ of the identity in $\Gamma_B$. 
\end{enumerate}

Write the trigonometric polynomials $p_\lambda$ in standard form $p_\lambda(\gamma)=\sum\limits_{k=1}^{N_\lambda} a_k^{(\lambda)} \chi_k(\gamma)$ and let $h\in AP(\Gamma, X)$. Define the Bochner-Fejer polynomial,
\begin{align*}
    \sigma_\lambda(h)(\gamma)&:=(p_\lambda \ast h)(\gamma)=\lim_{n\to\infty} \frac{1}{|H_n|}\sum_{\gamma'\in H_n} p_\lambda(\gamma-\gamma')h(\gamma')\\&=\lim_{n\to\infty} \frac{1}{|H_n|}\sum_{\gamma'\in H_n} \sum\limits_{k=1}^{N_\lambda}a_k^{(\lambda)} \chi_k(\gamma-\gamma')h(\gamma')
    \\&=\sum\limits_{k=1}^{N_\lambda}a_k^{(\lambda)} \hat{h}(\chi_k)\chi_k(\gamma).
\end{align*}
It is evident from the above calculation that $\Lambda(\sigma_\lambda(h))\subseteq \Lambda(h)$.
The following fundamental theorem is now obtained by a standard approximation argument (see for example \cite[Theorem 7.5]{bks}).

\begin{Theorem}
\label{CesApprox}
Let $h\in AP(\Gamma,X)$. Then
\[
 \sigma_\lambda(h)= p_\lambda \ast h \stackrel{\|\cdot\|_\infty}{\longrightarrow} h.
\]
\end{Theorem}

%%%%%%%%%%%%%%%%%%%%%%%%%%%%%%%%%%%%%%%%%%%%%%%%%%%%

\subsection{Approximation lemmas}
\label{s:approx}
%In this section we characterise when a $\Gamma$-symmetric framework admits no nontrivial almost periodic infinitesimal flexes  and in this case we say that it is almost periodically rigid. 
Let $\Gamma$ be a countable discrete abelian group and let $X$ be a finite dimensional complex Hilbert spaces.
Define a   representation $\pi_X:\Gamma\to B(\ell^\infty(\Gamma,X))$  % and $\pi_Y:\Gamma\to B(\ell^\infty(\Gamma,Y))$ 
where,
\[(\pi_X(\gamma)f)(\gamma')= f(\gamma'-\gamma),\quad \forall\,f\in \ell^\infty(\Gamma,X).\]
%\[(\pi_Y(\gamma)f)(\gamma')=f(\gamma'-\gamma),\quad \forall\,f\in \ell^\infty(\Gamma,Y).\]
%\[(W_{\gamma}g)(\gamma')=g(\gamma'-\gamma),\quad \forall\,g\in \ell^p(\Gamma,Y).\]
For $\chi \in \hat{\Gamma}$ and a finite subset $H\subset \Gamma$, define the bounded operator $U_X(\chi,H)\in B(\ell^\infty(\Gamma, X))$ %and $W(\chi,H)\in B(\ell^\infty(\Gamma, Y))$ 
by,
\[
U_X(\chi,H)=\frac{1}{|H|}\sum_{\gamma\in H} \chi(\gamma)\pi_X(\gamma)
\]
%\[
%W(\chi,H)=\frac{1}{|H|}\sum_{\gamma\in H} \chi(\gamma)\pi_Y(\gamma)\]
%For $\chi \in \hat{\Gamma}$ and $a\in X$, let $z(\chi,\tau, a)\in \ell^\infty(\Gamma,X)$ denote the bounded vector-valued function,
%\[z(\chi,\tau, a)(\gamma) =\chi(\gamma)\tau(\gamma)a.\]

\begin{Lemma}\label{L0}  
Let $\chi\in\hat{\Gamma}$ and let $a\in X$.
If $(H_n)$ is a Bohr-Bochner sequence for $\Gamma$ then the sequence $U_X(\chi',H_n)(\chi\otimes a)$
converges uniformly to $\chi\otimes a$ when $\chi' = \chi$ and to  zero otherwise.
\end{Lemma}

\begin{proof}
For each $\omega\in \Gamma$,
\begin{eqnarray*}
U_X(\chi',H_n)(\chi\otimes a)\omega 
&=&  \frac{1}{|H_n|}\sum_{\gamma\in H_n} \chi'(\gamma)\pi_X(\gamma)(\chi\otimes a)\omega  \\ 
&=&  \frac{1}{|H_n|}\sum_{\gamma\in H_n} \chi'(\gamma) \chi(\omega-\gamma)a\\
%&=&  \frac{1}{|H_n|}\sum_{\gamma\in H_n} \overline{\chi'(\gamma)}\chi(\omega+\gamma)\tilde{\tau}(-\gamma)(a)\\
&=&\left(  \frac{1}{|H_n|}\sum_{\gamma\in H_n} \chi'(\gamma)\overline{\chi(\gamma)}\right)\chi(\omega) a 
%&=&\left( \int_{\gamma\in \Gamma} \overline{\chi'(\gamma)}\chi(\gamma)\,d\mu\right)\chi(\omega)\tilde{\tau}(-\gamma)a_\chi
\end{eqnarray*}
The result now follows from Lemma \ref{sumchar}.
\end{proof}

\begin{Lemma}\label{L1}  
Let $h\in AP(\Gamma,X)$ and let $\chi\in\hat{\Gamma}$. If $(H_n)$ is a Bohr-Bochner sequence for $\Gamma$ then the sequence $U_X(\chi,H_n)h$
converges uniformly to $\chi\otimes \hat{h}(\chi)$.
\end{Lemma}
\begin{proof}
First, if $h=\chi\otimes a$ for some $\chi\in\hat{\Gamma}$ and  $a\in X$, then the result holds by Lemma \ref{L0}. By linearity the statement is true for every trigonometric polynomial $p\in AP(\Gamma,X)$. In the general case, let $h\in AP(\Gamma,X)$ and $\varepsilon>0$. Then there exists a trigonometric polynomial $p$, such that $\|h-p\|_\infty<\varepsilon.$
Since the maps $U_X(\chi,H_n)$ and $h\to \hat{h}$ are contractive, we have for large enough $n$ 
\begin{eqnarray*}
\|U_X(\chi,H_n)h-\chi\otimes \hat{h}(\chi)\|_\infty&\leq&
\|U_X(\chi,H_n)h-U_X(\chi,H_n)p\|_\infty+\|U_X(\chi,H_n)p-\chi\otimes \hat{h}(\chi)\|_\infty\\&\leq&
\|U_X(\chi,H_n)h-U_X(\chi,H_n)p\|_\infty+\|\chi\otimes \hat{p}(\chi)-\chi\otimes \hat{h}(\chi)\|_\infty +\varepsilon
\\&\leq&
2\|h-p\|_\infty +\varepsilon< 3\varepsilon
\end{eqnarray*}
\end{proof}

%%%%%%%%%%%%%%%%%%%%%%%%%%%%%%%%%%%%%%%%%%%%%%%%%%%%

\subsection{Intertwiners}
\label{s:intertwiners}
Let $X$ and $Y$ be finite dimensional complex Hilbert spaces.

\begin{Lemma}\label{L2}
Let $h\in AP(\Gamma,X)$ and let $C$ be an intertwiner for the representations $\pi_X$ and $\pi_Y$.
Then the following are equivalent:
\begin{enumerate}[(i)]
    \item $C(h)=0$.
    \item $C(\chi\otimes \hat{h}(\chi))=0$ for each $\chi\in\hat{\Gamma}$.
    \item $C(\sigma_\lambda(h))=0$ for each $\lambda$.
\end{enumerate}
\end{Lemma}

\begin{proof}
$(i)\Rightarrow(ii)$ Suppose $C(h)=0$ and let $\chi\in\hat{\Gamma}$.
Let $(H_n)$ be a Bohr-Bochner sequence for $\Gamma$. 
Then by Lemma \ref{L1},
\[
C(\chi\otimes \hat{h}(\chi))  = \lim_{n\to \infty}C(U_X(\chi,H_n)h) =\lim_{n\to \infty}U_Y(\chi,H_n)C(h)=0.
\] 

$(ii)\Rightarrow (iii)$ This follows from the linearity of $C$.

$(iii)\Rightarrow (i)$ This follows from Theorem \ref{CesApprox} and the continuity of $C$.
\end{proof}

Let $U(X)$ denote the unitary group of the Hilbert space $X$ and let $\tau:\Gamma\to U(X)$ be a unitary representation. Define the isometric isomorphism $T_\tau\in B(\ell^\infty(\Gamma,X))$ where, for each $f\in \ell^\infty(\Gamma,X)$,
\[T_\tau(f)(\gamma) = \tau(\gamma)f(\gamma).\]
Define the representation $\pi_{X,\tau}:\Gamma\to B(\ell^\infty(\Gamma,X))$ where,
\[\pi_{X,\tau}(\gamma) = T_{\tau}\circ \pi_X(\gamma)\circ T_\tau^{-1}.\]
We refer to a function of the form $T_\tau h$, where $h\in AP(\Gamma,X)$, as a {\em twisted almost periodic function}. 

\begin{Lemma}\label{L3}
Let $h\in AP(\Gamma,X)$ and let $C$ be an intertwiner for the representations $\pi_{X,\tau}$ and $\pi_Y$.
Then the following are equivalent:
\begin{enumerate}[(i)]
    \item $C(T_\tau h)=0$.
    \item $CT_\tau(\chi\otimes \hat{h}(\chi))=0$ for each $\chi\in\hat{\Gamma}$.
    \item $CT_\tau(\sigma_\lambda(h))=0$ for each $\lambda$.
\end{enumerate}
\end{Lemma}

\begin{proof}
Note that $CT_\tau$ is an intertwiner for the representations $\pi_X$ and $\pi_Y$. The lemma now follows directly from Lemma \ref{L2}.
\end{proof}

%%%%%%%%%%%%%%%%%%%%%%%%%%%%%%%%%%%%%%%%%%%%%%%%%%%%

\subsection{Twisted almost-periodic flexes}
Let $\G_0=(G_0,m,\varphi,\tau)$ be a $\Gamma$-gain framework.
Recall, from Section \ref{s:reps}, that $\tilde{\tau}:\Gamma\to B(X^{V_0})$ denotes the unitary representation, 
$$\tilde{\tau}(\gamma)(x_v)_{v\in V_0} = (d\tau(\gamma)x_v)_{v\in V_0}$$
and $T_{\tilde{\tau}}\in B(\ell^\infty(\Gamma,X^{V_0}))$ is the isometric isomorphism,
$$T_{\tilde{\tau}}(f)(\gamma) = \tilde{\tau}(\gamma)f(\gamma).$$
A {\em twisted almost periodic flex} of $\G_0$  is a twisted almost periodic function  $T_{\tilde{\tau}}h\in \ell^\infty(\Gamma,X^{V_0})$ which is also a bounded infinitesimal flex of $\G_0$. 
We denote by $\T(\G_0)$ the Banach space of twisted almost periodic flexes of $\G_0$.
It follows from Theorem \ref{t:transpace} that every translation of $\G_0$ is a twisted almost periodic flex of $\G_0$.
Note  that $\chi$-symmetric flexes of $\G_0$ are also  twisted almost periodic flexes of $\G_0$.

We now prove the two main results in this section.

\begin{Theorem}
\label{t:bohr}
Let $\G_0=(G_0,m,\varphi,\tau)$ be a $\Gamma$-gain framework. Then,
\[\Omega(\G_0) = \bigcup_{g\in\T(\G_0)} \Lambda(T_{\tilde{\tau}}^{-1}g).\]
Moreover, every $\chi$-symmetric vector of the form $z(\chi,\hat{h}(\chi))$, where $T_{\tilde{\tau}}h\in\T(\G_0)$, is a bounded infinitesimal flex of $\G_0$.
\end{Theorem}

\proof
Suppose $\chi\in \Omega(\G_0)$, so there exists non-zero $a\in \ker O_{\G_0}(\chi)$. 
Note that $\chi$ lies in the Bohr-Fourier spectrum $\Lambda(\chi\otimes a)$.
By Lemma \ref{l:kernel}, $z(\chi,a)=T_{\tilde{\tau}}(\chi\otimes a)\in \T(\G_0)$. 

For the reverse inclusion, let $h=T_{\tilde{\tau}}^{-1}g$ for some $g\in \T(\G_0)$ and suppose $\chi\in\Lambda(h)$.
Then $\tilde{C}(\G_0)T_{\tilde{\tau}}h=0$.
By Proposition \ref{p:intertwine}, $\tilde{C}(\G_0)$ is an intertwiner for the representations $\pi_{X^{V_0},{\tilde{\tau}}}$ and $\pi_{Y^{E_0}}$.
Thus, by Lemma \ref{L3}, $z(\chi,\hat{h}(\chi))=T_{\tilde{\tau}}(\chi\otimes \hat{h}(\chi))\in \ker\tilde{C}(\G_0)$. By Proposition \ref{p:orbit}, $\hat{h}(\chi)\in \ker O_{\G_0}(\chi)$ where $\hat{h}(\chi)\not=0$. Hence, $\chi \in \Omega(\G_0)$.
\endproof

\begin{Theorem}
\label{t:apr}
Let $\mathcal{G}_0=(G_0,m,\varphi,\tau)$ be a $\Gamma$-gain framework.
The following statements are equivalent.
\begin{enumerate}[(i)]
\item
Every twisted almost periodic flex of $\G_0$ is a translation of $\G_0$.
\item  $\Omega(\G_0)=\Omega_{js}(\tau)$ and every non-zero $\chi$-symmetric flex of $\G_0$ is a translation of $\G_0$. 
\end{enumerate}
\end{Theorem}

\proof
$(i)\Rightarrow (ii)$
Every  non-zero $\chi$-symmetric flex of $\G_0$ lies in $\T(\G_0)$, and hence is a translation of $\mathcal{G}_0$ by $(i)$.  If $\chi\in \Omega(\G_0)$ then, by Lemma \ref{l:kernel}, there exists $a$ such that $z(\chi,a)\in\ker \tilde{C}(\G_0)$. By the hypothesis, $z(\chi,a)$ is a translation of $\G_0$. By Theorem \ref{t:transpace}, $z(\chi,a)$ is a joint spectral $\chi$-symmetric vector for $\G_0$. Thus $\chi\in \Omega_{js}(\tau)$.

$(ii)\Rightarrow (i)$
Let $T_{\tilde{\tau}} h$ be a twisted almost periodic flex of $\G_0$. 
By Theorem \ref{t:bohr}, $\Lambda(h)\subseteq\Omega(\G_0)=\Omega_{js}(\tau)$ and 
each  $\chi$-symmetric vector   $z(\chi,\hat{h}(\chi))$ lies in  $\ker \tilde{C}(\G_0)$.
Hence, by the hypothesis, each non-zero $\chi$-symmetric vector   $z(\chi,\hat{h}(\chi))$ is  a translation of $\G_0$. It follows from Theorem \ref{CesApprox} that $T_{\tilde{\tau}}h$ is a translation of $\G_0$ as desired.
\endproof

\begin{Remark}
Periodic bar-and-joint frameworks with full affine span in $\bR^d$ are represented  by  $\Gamma$-gain frameworks where $\Gamma=\bZ^d$  and $\tau(\bZ^d)$ is a translation group. Since $d\tau(\bZ^d)$ is the trivial group, it follows that $\Omega_{js}(\tau)=\{1_{\hat{\Gamma}}\}$. In the Euclidean case, Theorem \ref{t:apr} is an equivalent statement to \cite[Theorem 4]{bkp} and can be regarded as a characterisation of almost periodic rigidity for periodic bar-and-joint frameworks.

In \cite{kit-pow-2}, 
the authors showed that a bar-and-joint framework in $\ell_q^d$, with $q\not=2$, is rigid if and only if every infinitesimal flex is a translation. 
Hence Theorem \ref{t:apr} completely characterises  twisted almost periodic rigidity for frameworks in $\ell_q^d$, for $q\not=2$.
As seen in Examples \ref{Ex:2loopsA}, \ref{Ex:2loopsB} and later in Example \ref{Ex:2loopsAinfty} it is possible for an almost periodic flex to arise from a rotation of the framework. Thus the question of characterising twisted almost periodic rigidity in general is more subtle.

It is currently unknown if the existence of a non-translational bounded infinitesimal flex implies the existence of a non-translational $\chi$-symmetric flex.
\end{Remark}

%%%%%%%%%%%%%%%%%%%%%%%%%%%%%%%%%%%%%%%%%%%%%%%%%%%%

\section{Further examples}
\label{s:FurtherExamples}
In this section, we apply the theory developed thus far to infinite symmetric bar-and-joint frameworks. In each case, the symmetry group for the bar-and-joint framework is a finitely generated discrete abelian group isomorphic to $\mathbb{Z}\times\mathbb{Z}_2$.
Firstly, we consider a  bounded bar-and-joint framework in $3$-dimensional Euclidean space which exhibits $C_{\infty h}$-symmetry. Here $C_{\infty h}$ is a symmetry group generated by an irrational rotation and a reflection in a plane which is orthogonal to the axis of rotation. We show that the space of bounded infinitesimal flexes is infinite dimensional and that rotations arise as twisted almost periodic flexes. Following this, we present a contrasting example of an unbounded bar-and-joint framework in the $\ell_q$-plane, for any $q\in(1,\infty)$, with $p11m$-symmetry. Here $p11m$ is the frieze group generated by a horizontal translation and reflection in a horizontal line. We show that the RUM spectrum consists of the joint spectral points and exactly one additional point. 

%%%%%%%%%%%%%%%%%%%%%%%%%%%%%%%%%%%%%%%%%%%%%%%%%%%%

\subsection{$C_{\infty h}$-symmetry}
\label{s:infinityh}
Let $G_0=(V_0,E_0)$ be the directed multigraph with a single vertex $v$ and two loops $e_1,e_2$. 
Let $\Gamma = \mathbb{Z}_2\times\mathbb{Z}$ and define $m:E_0\to\Gamma$ by setting  $m(e_1)=(0,1)$ and $m(e_2)=(1,1)$. 
The covering graph $G=(V,E)$ has vertex set $V=V_0\times\Gamma$. For convenience, we write $v_{0,0} := (v,(0,0))$ and $v_{l,m}:=(v,(l,m))$, for each $(l,m)\in \Gamma$. 
The dual group $\hat{\Gamma}$ is the multiplicative group of characters $\chi_{j,\omega}:\mathbb{Z}_2\times\mathbb{Z}\to\bT$, $\chi_{j,\omega}(l,m)=(-1)^{jl}\omega^{m}$, where $j\in\mathbb{Z}_2$ and $\omega \in \bT$.
   
 Define a group homomorphism $\tau:\Gamma\to \Isom(\mathbb{C}^3)$ by setting $\tau(1,0)$ to be the complexification of orthogonal reflection in the $xy$-plane and $\tau(0,1)$ to be the complexification of a clockwise irrational rotation by $\vartheta\in 2\pi\mathbb{R}\backslash \mathbb{Q}$ about the $z$-axis,
$$\tau(1,0) = \begin{bmatrix}
1 &0&0\\
0&1&0\\
0&0&-1
\end{bmatrix},
\qquad
\tau(0,1) = \begin{bmatrix}
\cos\vartheta&\sin\vartheta & 0\\
-\sin\vartheta&\cos\vartheta & 0\\
0&0&1
\end{bmatrix}.$$

\begin{Example}
    \label{Ex:2loopsinfty}
Let $\varphi=(\varphi_{e_1},\varphi_{e_2})$ be a pair of linear functionals from $\mathbb{C}^3$ to $\mathbb{C}$. The orbit matrix of the corresponding $\Gamma$-gain framework $\G_0=(G_0,m,\varphi,\tau)$ has the following form:
\begin{equation}
\label{Ex:2loopsOrbitinf}
O_{\G_0}(\chi_{j,\omega})
=\begin{bmatrix} 
\varphi_{e_1}\circ(I - \omega \tau(0,1)) \\
\varphi_{e_2}\circ(I - (-1)^j\omega\tau(1,1)) 
\end{bmatrix}
\end{equation}
It follows by the rank-nullity theorem that $\Omega(\G_0)=\hat{\Gamma}$. 
 Thus,  by Corollary \ref{c:infdim}, the  space   of bounded infinitesimal flexes of $\G_0$ is infinite dimensional.
 
Let $\eta=e^{i\vartheta}$ and let $T=(\tau(1,0),\tau(0,1))$. 
Note that $\sigma(T) = \{(-1,1),(1,\eta),(1,\bar{\eta})\}$.
To see this note that $\lambda_1=(-1,1)$ is a joint eigenvalue for $T$ with joint eigenvector $a_{\lambda_1}=(0,0,1)$, $\lambda_2=(1,\eta)$ is a joint eigenvalue for $T$ with joint eigenvector $a_{\lambda_2}=(-i,1,0)$ and $\lambda_3=(1,\bar{\eta})$ is a joint eigenvalue for $T$ with joint eigenvector $a_{\lambda_3}=(i,1,0)$. 
The corresponding characters satisfy, 
$$\bar{\chi}_{\lambda_1}(l,m) = (-1)^l=\chi_{1,1}(l,m),
\quad \bar{\chi}_{\lambda_2}(l,m) = \bar{\eta}^m=\chi_{0,\bar{\eta}}(l,m),
\quad \bar{\chi}_{\lambda_3}(l,m) = \eta^{m}=\chi_{0,\eta}(l,m).$$
 In particular,  the joint spectral points in the RUM spectrum are $\Omega_{js}(\tau)=\{\chi_{1,1},\chi_{0,\bar{\eta}},\chi_{0,\eta}\}$.
 
 By Lemma \ref{l:trivial}, the corresponding joint spectral $\chi$-symmetric vectors for $\G$ are the translations:
 \begin{itemize}
  \item $z(\bar{\chi}_{\lambda_1},a_{\lambda_1})(\gamma)=z(\chi_{1,1},a_{\lambda_1})(\gamma)=a_{\lambda_1}$.
    \item $z(\bar{\chi}_{\lambda_2},a_{\lambda_2})(\gamma)=z(\chi_{0,\bar{\eta}},a_{\lambda_2})(\gamma)=a_{\lambda_2}$.
    \item $z(\bar{\chi}_{\lambda_3},a_{\lambda_3})(\gamma)=z(\chi_{0,\eta},a_{\lambda_3})(\gamma)=a_{\lambda_3}$.
 \end{itemize}
By Theorem \ref{t:transpace}, the corresponding joint spectral $\chi$-symmetric vectors for $\G_0$ generate the $3$-dimensional translation space $Z_{\G_0}$. 
\end{Example}

Let $p:V\to \mathbb{R}^3$ be the assignment,
$$p(v_{0,0}) :=(-\sqrt{3},-1,1),\qquad
p(v_{l,m}) := \tau(l,m)p(v_{0,0}).$$
In particular, we have
\begin{align*}
p(v_{0,1}) &= \tau(0,1)p(v_{0,0})=(-\sqrt{3}\cos\vartheta-\sin\vartheta,\sqrt{3} \sin\vartheta-\cos\vartheta,1),\\ p(v_{1,1}) &= \tau(1,1)p(v_{0,0})=(-\sqrt{3}\cos\vartheta-\sin\vartheta,\sqrt{3} \sin\vartheta-\cos\vartheta,-1).
\end{align*}
Note that the image $p(V)$ of the vertex set of the covering graph $G$ is not discrete; it is dense in the set,
$$\{(x,y,z)\in\mathbb{R}^3:\|(x,y,0)\|=2, \,z=\pm 1\}.$$
In the following example, the linear transformations $\varphi=(\varphi_{e_1},\varphi_{e_2})$ are derived from the assignment $p$ and the Euclidean norm on $\mathbb{R}^3$ in the manner of Example \ref{Ex:frameworkA}.
 
\begin{Example}
\label{Ex:2loopsAinfty}
Let $\G_0=(G_0,m,\varphi,\tau)$ be the $\Gamma$-gain framework 
with the following linear transformations,
\begin{align*}
\varphi_{e_1}(a) &= (p(v_{0,0})-p(v_{0,1}))\cdot a =(-\sqrt{3}(1-\cos\vartheta)+\sin\vartheta,-\sqrt{3} \sin\vartheta-(1-\cos\vartheta),0)\cdot a,\\
\varphi_{e_2}(a) &= (p(v_{0,0})-p(v_{1,1}))\cdot a =(-\sqrt{3}(1-\cos\vartheta)+\sin\vartheta,-\sqrt{3} \sin\vartheta-(1-\cos\vartheta),2)\cdot a,
\end{align*}
% \varphi_{e_1}=[ -\sqrt{3} -3 0]^\intercal and \varphi_{e_2}= [ -\sqrt{3} -3 2]^
for all $a\in \mathbb{C}^3$.  
Using Equation \ref{Ex:2loopsOrbitinf}, the orbit matrices for $\G_0$ take the form,
 \begin{eqnarray*}
 O_{\G_0}(\chi_{j,\omega})  
 =  \left[\begin{smallmatrix} 
 \sqrt{3}(1+\omega)(\cos\vartheta -1)+(1-\omega)\sin\vartheta & &(1+\omega)(\cos\vartheta -1)-\sqrt{3}(1-\omega)\sin\vartheta & &0\\
 \sqrt{3}(1+(-1)^j\omega)(\cos\vartheta -1)+(1-(-1)^j\omega)\sin\vartheta & &
 (1+(-1)^j\omega)(\cos\vartheta -1)-\sqrt{3}(1-(-1)^j\omega)\sin\vartheta & & 2(1+(-1)^j\omega) 
 \end{smallmatrix}\right]
 \end{eqnarray*}
 
 The vector $u_{0,1}=(1,-\sqrt{3},0)$ spans the kernel of $O_{\G_0}(\chi_{0,1})$. The $\chi$-symmetric vector $z(\chi_{0,1},u_{0,1})$ describes an anticlockwise rotation of the covering framework about the $z$-axis.
 
 The vector $u_{1,\eta}=(i,1, 2e^{i\pi/3})$ spans the kernel of $O_{\G_0}(\chi_{1,\eta})$. The components of the $\chi$-symmetric vector $z(\chi_{1,\eta},u_{1,\eta})$ are,  
  $$z_{(l,m)}= \begin{bmatrix} i(-1)^l \\ (-1)^l \\ 2\eta^m e^{i\pi/3} \end{bmatrix},$$
  and so the real and imaginary parts of $z(\chi_{1,\eta},u_{1,\eta})$ are the twisted almost periodic flexes with components,
  $$\Real z_{(l,m)}= \begin{bmatrix} 0 \\ (-1)^l \\ 2\cos(\pi/3+m\vartheta) \end{bmatrix} \qquad \text{ and } \qquad
    \Imag z_{(l,m)}= \begin{bmatrix} (-1)^l \\ 0 \\ 2\sin(\pi/3+m\vartheta) \end{bmatrix}.$$
  The real part of $z(\chi_{1,\eta},u_{1,\eta})$ describes a rotation of the covering framework about the $x$-axis. The imaginary part of  $z(\chi_{1,\eta},u_{1,\eta})$ describes a rotation of the covering framework about the $y$-axis. 
  
 The vector $u_{1,\bar{\eta}}=(-i,1, 2e^{-i\pi/3})$ spans the kernel of $O_{\G_0}(\chi_{1,\bar{\eta}})$. The components of the $\chi$-symmetric vector $z(\chi_{1,\bar{\eta}},u_{1,\bar{\eta}})$ are,
 $$z_{(l,m)}= \begin{bmatrix} i(-1)^{l+1}\\(-1)^l \\ 2\bar{\eta}^me^{-i\pi/3}  \end{bmatrix},$$
 and so $z(\chi_{1,\bar{\eta}},u_{1,\bar{\eta}})$ is the twisted almost periodic flex whose components are the complex conjugates of the components of $z(\chi_{1,\eta},u_{1,\eta})$. Hence the real and  imaginary parts of $z(\chi_{1,\bar{\eta}},u_{1,\bar{\eta}})$ also describe rotations of the covering framework about the $x$-axis and $y$-axis respectively. 
  \end{Example}

%%%%%%%%%%%%%%%%%%%%%%%%%%%%%%%%%%%%%%%%%%%%%%%%%%%%

\subsection{Frieze group symmetry}
\label{s:frieze}
Let $G_0=(V_0,E_0)$ be the directed multigraph with  a single vertex $v$ and edge set $E_0=\{e_1,e_2\}$ where  $e_1$ and $e_2$ are both loops at $v$.  Let $\Gamma = \mathbb{Z}\times \mathbb{Z}_2$ and define $m:E_0\to\Gamma$ by setting  $m(e_1)=(1,0)$ and $m(e_2)=(1,1)$. 
  The $\Gamma$-gain graph $(G_0, m)$ and its covering graph are illustrated in Figure \ref{fig:diamondlattice}.
   The covering graph $G=(V,E)$ has vertex set $V=V_0\times\Gamma$. For convenience, we write $v_{m,j} := (v,(m,j))$ for each $m\in \mathbb{Z}$ and $j\in\mathbb{Z}_2$. 
   
 Define a group homomorphism $\tau:\Gamma\to \Isom(\mathbb{C}^2)$ by setting,
 \begin{equation*}
\tau(m,j)\begin{bmatrix}a_1\\a_2\end{bmatrix}
=\begin{bmatrix}
1 & 0  \\
0 & (-1)^j 
\end{bmatrix}\begin{bmatrix}a_1\\a_2\end{bmatrix} + \begin{bmatrix}
    m \\0
\end{bmatrix}, \quad
m\in\bZ,\, j\in \bZ_2.
\end{equation*}
The dual group $\hat{\Gamma}$ is the multiplicative group of characters $\chi_{\omega,\iota}:\Gamma\to\mathbb{T}$, $\chi_{\omega,\iota}(m,j)=\omega^{m}\iota^j$, where $\omega\in\mathbb{T}$ and $\iota\in\{-1,1\}$.

\begin{Example}
    \label{Ex:diamondlattice}
Let $\G_0=(G_0,m,\varphi,\tau)$ be a $\Gamma$-gain framework where $\varphi=(\varphi_{e_1},\varphi_{e_2})$ is an arbitrary pair of linear transformations from $\mathbb{C}^2$ to $\mathbb{C}$.
 For $\chi_{\omega,\iota}\in\hat{\Gamma}$ the corresponding orbit matrix is,
\begin{equation}
\label{Ex:dlOrbit}
O_{\G_0}(\chi_{\omega,\iota})
=\kbordermatrix{ &v_0\\
e_1& \varphi_{e_1}\circ(I-\omega d\tau(1,0))  \\  
e_2& \varphi_{e_2}\circ(I-\omega\iota d\tau(1,1))  }
\end{equation}
Let $T=(d\tau(1,0),d\tau(0,1))$ and note that $\sigma(T) = \{(1,1),(1,-1)\}$.
 To see this note that $\lambda_1=(1,1)$ is a joint eigenvalue for $T$ with eigenvector $a_{\lambda_1}=(1,0)$ and $\lambda_2 =(1,-1)$ is a joint eigenvalue for $T$ with   eigenvector $a_{\lambda_2}=(0,1)$. 
 The corresponding characters satisfy, 
 $$\bar{\chi}_{\lambda_1}(m,j) = 1=\chi_{1,1}(m,j),\quad\quad \bar{\chi}_{\lambda_2}(m,j) = (-1)^j=\chi_{1,-1}(m,j).$$ 
 In particular,  the RUM spectrum  $\Omega(\G_0)$ contains the joint spectral points $\Omega_{js}(\tau)=\{\chi_{1,1},\chi_{1,-1}\}$.
 
 By Lemma \ref{l:trivial}, the corresponding joint spectral $\chi$-symmetric vectors for $\G$ are the translations:
 \begin{itemize}
  \item $z(\bar{\chi}_{\lambda_1},a_{\lambda_1})(\gamma)=z(\chi_{1,1},a_{\lambda_1})(\gamma)=a_{\lambda_1}$.
    \item $z(\bar{\chi}_{\lambda_2},a_{\lambda_2})(\gamma)=z(\chi_{1,-1},a_{\lambda_2})(\gamma)=a_{\lambda_2}$.
 \end{itemize}
 As shown in Theorem \ref{t:transpace}, the above joint spectral $\chi$-symmetric vectors for $\G_0$ span the $2$-dimensional translation space $Z_{\G_0}$.
\end{Example}

For the linear constraints, we consider  linear functionals that arise from an $\ell_q$-norm on $\mathbb{R}^2$, where $q\in(1,\infty)$. For more details, see \cite{kit-pow-2}.
Let $p:V\to \mathbb{R}^2$ be the assignment given by,
\[p_{m,j}:=p(v_{m,j})=\begin{bmatrix}m\\(-1)^{j+1}\end{bmatrix},\quad \forall\,m\in \bZ,\,j\in\{0,1\}.\] 
See the right hand image in Figure \ref{fig:diamondlattice} for an illustration.
Let us write $d_{e_1}=p(v_{0,0})-\tau(1,0)p(v_{0,0})$ and
$d_{e_2}=p(v_{0,0})-\tau(1,1)p(v_{0,0})$.
For $a=(a_1,a_2)\in\mathbb{R}^2$ we write,
$$a^{(k)} = (\sgn(a_1)|a_1|^k,\, \sgn(a_2)|a_2|^k).$$

\begin{Example}
 \label{Ex:diamondlatticeA}
Let $\G_0=(G_0,m,\varphi,\tau)$ be the $\Gamma$-gain framework 
with linear transformations $\varphi=(\varphi_{e_1},\varphi_{e_2})$ given by,
\begin{align*}
\varphi_{e_1}(a) &=\|d_{e_1}\|_q^{1-q}d_{e_1}^{(q-1)}\cdot a =(-1, 0)\cdot a,\\
\varphi_{e_2}(a) &= \|d_{e_2}\|_q^{1-q}d_{e_2}^{(q-1)}\cdot a =(1+2^q)^{\frac{1-q}{q}}(-1, -2^{q-1})\cdot a,
\end{align*}
for all $a\in \mathbb{C}^2$. The linear functionals $\varphi=(\varphi_{e_1},\varphi_{e_2})$ are derived from the $\ell_q$-distance constraints on the covering framework. 
The orbit matrices have the following form,
\begin{eqnarray}
\label{Ex:diamondp}
O_{\G_0}(\chi_{\omega,\iota})
&=& \begin{bmatrix}
-(1-\omega)& 0  \\  
-(1+2^q)^{\frac{1-q}{q}}(1-\omega \iota)& -(1+2^q)^{\frac{1-q}{q}}2^{q-1}(\omega \iota+1) 
\end{bmatrix}
\end{eqnarray}

Note that the RUM spectrum $\Omega(\G_0)$ consists of the joint spectral $\chi$-symmetric vectors $\chi_{1,1}$ and $\chi_{1,-1}$ (as shown in Example \ref{Ex:diamondlattice}), and exactly one additional point $\chi_{-1,1}$. It is trivial to see that the kernels of $O_{\G_0}(\chi_{1,1})$ and $O_{\G_0}(\chi_{1,-1})$ are spanned by the joint eigenvectors for the joint eigenvalues $\lambda_1$ and $\lambda_2$ respectively, given in Example \ref{Ex:diamondlattice}.

The vector $u_{-1,1}=(0,1)$ spans the kernel of $O_{\G_0}(\chi_{-1,1})$. The $\chi$-symmetric vector $z(\chi_{-1,1},u_{-1,1})$ is given by,
  $$z_{(m,j)}= \begin{bmatrix} 0\\ (-1)^{m+j}  \end{bmatrix}, \, m\in \bZ, j\in\{-1,1\}.$$
The $\chi$-symmetric vector $z(\chi_{-1,1},u_{-1,1})$ describes an alternating periodic motion of the covering framework.
By Theorem \ref{t:bohr}, it follows that the space of twisted almost periodic flexes of $\G_0$ has dimension $3$. 
\end{Example}

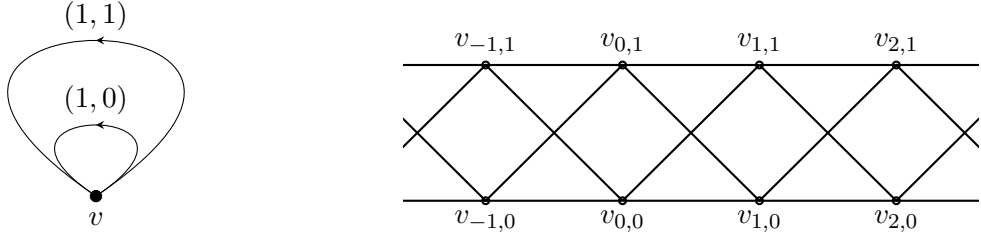
\begin{figure}[h!]
\centering
\hspace{-20mm}
\begin{tikzpicture}[scale=0.6]
    % Define the single small vertex labeled v
    \node[circle, draw, fill=black, inner sep=1.5pt, label=below:$v$] (v) at (0,2) {};
     \draw[postaction={decorate,decoration={markings,mark=at position 0.5 with {\arrow{stealth}}}}]
        (v) to[out=35, in=145, looseness=90] node[midway, above] {$(1,1)$} (v);
    \draw[postaction={decorate,decoration={markings,mark=at position 0.5 with {\arrow{stealth}}}}] 
        (v) to[out=35, in=145, looseness=40] node[midway, above] {$(1,0)$} (v);
\end{tikzpicture}
\begin{tikzpicture}[scale=0.9]
\clip (-3.2,-0.5) rectangle (5.2cm, 4cm);

\draw [thick](-7,0)--(-0,0);
\draw [thick](2,0)--(7,0);
\draw [thick](-7,2)--(2,2);
\draw [thick](2,2)--(7,2);
\draw [thick](-7,1)--(-6,2)--(-4,0)--(-2,2)--(0,0);
\draw [thick](2,2)--(4,0)--(6,2)--(7,1);
\draw [thick](-7,1)--(-6,0)--(-4,2)--(-2,0)--(0,2)--(2,0)--(4,2)--(6,0)--(7,1);
\draw [thick] (-6,0) circle [radius=0.05];
\draw [thick] (-4,0) circle [radius=0.05];
\draw [thick] (-2,0) circle [radius=0.05];
\draw [thick](0,0)--(2,2);
\draw [thick](0,0)--(2,0);
\draw [thick] (0,0) circle [radius=0.05];
\draw [thick] (2,0) circle [radius=0.05];
\draw [thick] (4,0) circle [radius=0.05];
\draw [thick] (6,0) circle [radius=0.05];
\draw [thick] (-6,2) circle [radius=0.05];
\draw [thick] (-4,2) circle [radius=0.05];
\draw [thick] (-2,2) circle [radius=0.05];
\draw [thick] (0,2) circle [radius=0.05];
\draw [thick] (2,2) circle [radius=0.05];
\draw [thick] (4,2) circle [radius=0.05];
\draw [thick] (6,2) circle [radius=0.05];
\node [below] at (0,0) {$v_{0,0}$};
\node [below] at (2,0) {$v_{1,0}$};
\node [below] at (4,0) {$v_{2,0}$};
\node [below] at (-2,0) {$v_{-1,0}$};
\node [above] at (0,2) {$v_{0,1}$};
\node [above] at (2,2) {$v_{1,1}$};
\node [above] at (4,2) {$v_{2,1}$};
\node [above] at (-2,2) {$v_{-1,1}$};
\end{tikzpicture}
\caption{The  gain graph (left) and  covering graph (right) in Section \ref{s:frieze}.}\label{fig:diamondlattice}
\end{figure}

\begin{Example}
Let us consider now the directed multigraph graph $G_1=G_0+e_3$, where $e_3$ is again a loop at $v$. We also extend the gain function $m$ above, setting $m(e_3)=(2,1)$. Let  $d_{e_3}=p(v_{0,0})-\tau(2,1)p(v_{0,0})$ and define,
$$\varphi_{e_3}(a) = \|d_{e_3}\|_q^{1-q}d_{e_3}^{(q-1)}\cdot a =2^{\frac{1}{q}-1}(-1, -1)\cdot a.$$
Note that the extra edge admits a third row for the orbit matrix of $\G_1=(G_1,m,\varphi,\tau)$,  
\begin{equation*}
\label{ex:third_row}
O_{\G_1}(\chi_{\omega,\iota})(e_3)
=\begin{bmatrix}
    \varphi_{e_3}\circ(I-\omega^2\iota d\tau(2,1))
\end{bmatrix} 
\end{equation*}
Hence the orbit matrices for $\G_1$ take the form,
\begin{eqnarray*}
\label{Ex:diamondppluse3}
O_{\G_1}(\chi_{\omega,\iota})
&=& \begin{bmatrix}-(1-\omega)& 0  \\  
-(1+2^q)^{\frac{1-q}{q}}(1-\omega \iota)& -(1+2^q)^{\frac{1-q}{q}}2^{q-1}(\omega \iota+1) \\
-2^{\frac{1}{q}-1}(1-\omega^2 \iota)& -2^{\frac{1}{q}-1}(\omega^2 \iota+1)\\
\end{bmatrix}
\end{eqnarray*}
It is evident that every vector in $\ker(O_{\G_1}(\chi_{\omega,\iota}))$ lies also in $\ker(O_{\G_0}(\chi_{\omega,\iota}))$, and that the two $\Gamma$-gain frameworks share the same joint spectral $\chi$-symmetric vectors. However, one can see that
$\Omega(\G_1)=\Omega_{js}(\tau)$. Moreover, every non-zero $\chi$-symmetric flex for $\G_1$ is a joint spectral $\chi$-symmetric vector for $\G_1$. It follows by Theorem \ref{t:apr} that every almost periodic flex of $\G_1$ is a translation. 
\end{Example}

%%%%%%%%%%%%%%%%%%%%%%%%%%%%%%%%%%%%%%%%%%%%%%%

\end{document}